\documentclass[12pt,reqno]{amsart}

\usepackage{amssymb,longtable}

\newcommand{\e}{\varepsilon}
\newcommand{\om}{\omega}

\newcommand{\M}{\mathcal{M}}
\newcommand{\la}{\lambda}

\newcommand{\fy}{\varphi}

\newcommand{\p}{\partial}
\newcommand{\I}{\infty}
\newcommand{\wt}[1]{\widetilde {#1}}
\newcommand{\ti}{\widetilde}
\newcommand{\R}{\mathbb{R}}

\newcommand{\K}{\mathcal{K}}
\newcommand{\UU}{\mathcal{U}}
\renewcommand{\S}{\mathcal{S}}
\newcommand{\T}{\mathcal{T}}

\renewcommand{\Im}{\mathop{\mathrm{Im}}}
\renewcommand{\bar}{\overline}
\renewcommand{\hat}{\widehat}

\renewcommand{\r}{\rho}

\numberwithin{equation}{section}

\newtheorem{thm}{Theorem}[section]

\newtheorem{lem}[thm]{Lemma}
\newtheorem{prop}[thm]{Proposition}

\theoremstyle{remark}

\newcommand{\ran}{\rangle}
\newcommand{\lan}{\langle}

\newcommand{\lec}{\lesssim}
\newcommand{\gec}{\gtrsim}

\newcommand{\EQ}[1]{\begin{equation} \begin{split} #1 \end{split} \end{equation}}
\newcommand{\Del}[1]{}
\newcommand{\CAS}[1]{\begin{cases} #1 \end{cases}}
\newcommand{\pt}{&}
\newcommand{\pr}{\\ &}
\newcommand{\pq}{\quad}
\newcommand{\pn}{}

\newcommand{\LR}[1]{{\lan #1 \ran}}
\newcommand{\de}{\delta}
\newcommand{\si}{\sigma}

\renewcommand{\t}{\tau}

\newcommand{\ka}{\kappa}
\newcommand{\ga}{\gamma}
\newcommand{\x}{\xi}

\newcommand{\na}{\nabla}

\renewcommand{\th}{\theta}

\newcommand{\B}{\mathcal{B}}
\newcommand{\De}{\Delta}

\newcommand{\Ga}{\Gamma}

\newcommand{\sg}{\mathfrak{s}}
\newcommand{\Sg}{\mathfrak{S}}

\newcommand{\Cu}{\bigcup}
\newcommand{\Ca}{\bigcap}
\newcommand{\HH}{\mathcal{H}}
\def\dist{\mathrm{dist}}
\def\sign{\mathrm{sign}}

\def\nn{}

\author{K.~Nakanishi}
\address{Department of Mathematics, Kyoto University\\ Kyoto 606-8502, Japan}
\email{n-kenji@math.kyoto-u.ac.jp}

\author{W.~Schlag}
\address{Department of  Mathematics, The University of Chicago\\ Chicago, IL 60615, U.S.A.} 
\email{schlag@math.uchicago.edu}

\title[Global dynamics above the ground state energy]{Global dynamics above the ground state energy\\ for the focusing nonlinear Klein-Gordon equation}

\begin{document}




\begin{abstract}
The analysis of global dynamics of nonlinear dispersive equations has a long history starting from small solutions. In this paper we study the
focusing, cubic, nonlinear Klein-Gordon equation 
in $\R^{3}$ with large radial data in the energy space. This equation admits a unique positive stationary solution $Q$, called the ground state. 
In 1975 Payne and Sattinger showed that solutions $u(t)$ with energy $E[u,\dot u]$ {\em strictly below} that of the ground state are divided into two classes, depending on a suitable functional $K(u)$: If $K(u)<0$, then one has finite time  blow-up, if $K(u)\ge 0$ global existence; moreover, these sets are invariant under the flow. 
Recently, Ibrahim, Masmoudi and the first author~\cite{IMN} improved this result by establishing scattering to zero for $K[u]\ge 0$ by means of a variant of the Kenig-Merle method~\cite{KM1}, \cite{KM2}. 
In this paper we go slightly beyond the ground state energy and we give a complete description of the evolution in that case. For example, in a small neighborhood of $Q$ one encounters the following trichotomy: on one side of a center-stable manifold one has finite-time blow-up for $t\ge0$, on the other side scattering to zero, and on the manifold itself one has scattering to~$Q$, both as $t\to+\infty$. In total, the class of data with energy at most slightly above that of $Q$ is divided into nine disjoint nonempty sets each displaying different asymptotic behavior as $t\to\pm\infty$, which includes solutions blowing up in one time direction and scattering to zero on the other. The 
analogue of the solutions found by Duyckaerts, Merle~\cite{DM1}, \cite{DM2} for the energy critical wave and Schr\"odinger equations  appear here as the unique 
one-dimensional stable/unstable manifolds approaching~$\pm Q$ exponentially as $t\to\I$ or $t\to-\I$, respectively. 
The main technical ingredient in our proof is a ``one-pass'' theorem which excludes the existence of (almost) homoclinic orbits between $Q$ (as well as $-Q$) and (almost) heteroclinic orbits connecting $Q$ with $-Q$. In a companion
paper~\cite{NakS} we establish analogous properties for the NLS equation. 
\end{abstract}

\maketitle

\tableofcontents

\section{Introduction}
In this paper we study the global behavior of general solutions to the nonlinear Klein-Gordon equation (NLKG) with the focusing cubic nonlinearity on $\R^3$, i.e., 
\EQ{\label{eq:NLKG}
 \ddot u - \De u + u = u^3, \pq u(t,x):\R^{1+3}\to\R,}
which conserves the energy
\EQ{\nn \label{def E}
 E(\vec u):=\int_{\R^3} \big[\frac{|\dot u|^2+|\na u|^2+|u|^2}{2}-\frac{|u|^4}{4}\big]\, dx.}
We regard $H^1(\R^3)\times L^2(\R^3)$ as the phase space for this infinite dimensional Hamiltonian system. In other words, we write the solutions as 
\EQ{\nn \label{def H}
 \vec u(t):=(u(t),\dot u(t)) \in \HH:= H^1_{\mathrm{rad}}\times L^2_{\mathrm{rad}}.} 
There exists a vast literature on the wellposedness theory for this equation in the energy space since J\"orgens \cite{Jorgens} and Segal \cite{Segal}, as well as the scattering theory for small data for the focusing nonlinearity as in \eqref{eq:NLKG} and for large data for the defocusing equation; see Brenner~\cite{Bren1}, \cite{Bren2}, Ginibre, Velo~\cite{GV1}, \cite{GV2}, Morawetz, Strauss \cite{MoSt}, and Pecher~\cite{Pech}. In this paper, the scattering of a solution $u$ to a static state $\fy$ refers to the following asymptotic behavior: there exists a solution $v$ of the free Klein-Gordon equation such that 
\EQ{
 \|\vec u(t)-\vec\fy-\vec v(t)\|_\HH \to 0,}
as either $t\to\I$ or $t\to-\I$ (depending on the context). See also Strauss~\cite{Strauss} and~\cite{IMN} for a review of Strichartz estimates and wellposedness, as well as scattering in this setting.  

It is well-known \cite{Strauss77}, \cite{BerLions}, \cite{Coff} that there exists a unique radial positive ground state $Q(x)$, solving the static equation 
\EQ{\nn \label{eq Q}
 -\De Q + Q = Q^3,}
with the least energy 
\EQ{\nn \label{def J}
 E(Q)=J(Q):=\int_{\R^3} \big[\frac{|\na Q|^2+|Q|^2}{2}-\frac{|Q|^4}{4}\big] \,dx>0,}
among the static solutions, and that the solutions $u$ below the ground energy 
\EQ{\nn 
 E(\vec u)<E(Q)}
are split into two classes by the functional 
\EQ{\label{eq:K0}
 K_0(u):=\int_{\R^3}[|\na u|^2+|u|^2-|u|^4] \, dx=\p_\la|_{\la=1}J(\la u),}
\begin{enumerate}
\item If $K_0(u(0))\ge 0$ then the solution $u$ exists globally on $t\in\R$.
\item If $K_0(u(0))<0$ then the solution $u$ blows up both in $\pm t>0$. 
\end{enumerate}
In this paper, blow-up of a solution $u$ means that it cannot be extended beyond a finite time $T^*$ in the energy space $\HH$. Then the 
wellposedness theory implies that $\|\vec u(t)\|_\HH\to\I$ as $t$ approaches $T^*$. 

The above result of dichotomy is essentially due to Payne, Sattinger \cite{PS}, which of course implies that the ground state $Q$ is unstable. 
For a general theory of orbital stability vs.~instability of solitary wave solutions for equations of this type, see Grillakis, Shatah, Strauss~\cite{GSS1}, \cite{GSS2}. 

Recently, Kenig, Merle \cite{KM1,KM2} considered the focusing nonlinear Schr\"odinger equation (NLS) and the nonlinear wave equations with the energy-critical nonlinearity in three dimensions (as well as others)
 and obtained a stronger version of the dichotomy, by adding scattering to zero in the global existence scenario. 
 It has been extended to several similar equations, including the cubic focusing NLS equation by Holmer, Roudenko~\cite{HR}, and NLKG such as \eqref{eq:NLKG} in~\cite{IMN}. 

On the other hand, the second author~\cite{S}, followed by Beceanu~\cite{Bec}, constructed center-stable manifolds with finite codimensions around the ground state in the case of the focusing cubic NLS equation in three dimensions. 
A construction of such manifolds for all $L^{2}$ supercritical NLS equations in one dimension was given by Krieger and the second author~\cite{KrS1}, and for 
the radial critical wave equation in~$\R^{3}$ this was done in~\cite{KrS2}. 
Stable, unstable and center manifolds have been known to play an important role in hyperbolic dynamics for a long time, see the classical work by Hirsch, Pugh, Shub~\cite{HPS} as well as
many other works since then, such as Bates, Jones~\cite{BJ1}, \cite{BJ2}, Vanderbauwhede~\cite{V}, and Li, Wiggins~\cite{LiW}. 
While the more ODE oriented approach in Bates, Jones to PDEs such as~\eqref{eq:NLKG}  does construct center-stable and unstable local manifolds near an
unstable equilibrium in the energy space, no statement can be made about the trajectories once they leave a small neighborhood of the equilibrium. In other words, they are local-in-time results which do not involve any dispersive analysis. In contrast, the emphasis in~\cite{S}, \cite{KrS1}, \cite{KrS2}, \cite{Bec} lies with the asymptotic behavior of the solutions as $t\to \I$
and scattering to a suitable ``soliton'' is established in each of these works for data lying on a center-stable manifold modulo the group actions which preserve the equation. 
In~\cite{Bec} both points of view are united by establishing the existence of the manifold $\M$ from~\cite{S} in the energy class (in fact, in the critical $\dot H^{\frac12}$ topology) for the cubic NLS equation, with the added feature of global-in-time invariance of~$\M$ as $t\to\infty$, as well as scattering  in the energy class to~$Q$ modulo the symmetries for solutions starting on~$\M$. 

Note that those solutions scattering to $Q$ must have energy above the ground state, and so the solution sets of Payne, Sattinger on the one hand, and those on the center-stable manifolds on the other hand,  are necessarily disjoint from each other. 
More recently, Duyckaerts, Merle \cite{DM1,DM2} investigated the solutions on the threshold energy $E(\vec u)=E(Q)$ for energy critical equations, and proved that there are exactly two new solutions modulo symmetry; scattering to $Q$ as $t\to-\I$, while either scattering to $0$ or blowing up as $t\to\I$. 
They can be regarded as minimal solutions on the manifold~$\M$. 
This result is also extended to the cubic NLS by Duyckaerts, Roudenko \cite{DR}.  In this paper we exhibit these solutions as the unique stable/unstable manifolds associated with $\pm Q$, see~\cite{BJ2} for the definition of these objects. 

However, all of these works describe only part of the dynamics for energies near that of the ground state. 
A natural question to ask is whether the solutions near the soliton are separated by $\M$ into a region of scattering to zero, and the other region of blow-up. 
This is partially motivated by the study of ``critical phenomena'' in the physics literature, see for example 
Choptuik~\cite{Chop}, Choptuik, Chmaj, Bizon~\cite{ChopBiz},  and Bizon, Chmaj, Tabor~\cite{Biz}. 
For the $|u|^{5}$ wave equation in~$\R^{3}$,  blow-up for all data on one side of the tangent space to the center-stable manifold constructed in~\cite{KrS2}
was shown by 
Karageorgis, Strauss~\cite{KaS}. 

Our goal in this paper is to give a complete picture of the dynamics of~\eqref{eq:NLKG} for radial energy data with energy slightly larger than that of~$Q$: 
\EQ{ \label{def He}
 \HH^\e:=\{\vec u\in \HH \mid E(\vec u)<E(Q)+\e^2\}.}
Note that the only symmetry in $\HH$ is $u\mapsto-u$ in this setting.

\begin{thm}\label{thm:main}
Consider all solutions of NLKG~\eqref{eq:NLKG} with initial data $\vec u(0)\in\HH^\e$ for some small $\e>0$. We prove that the solution set is decomposed into nine non-empty sets characterized as 
\begin{enumerate}
\item Scattering to $0$ for both  $t\to\pm\I$, 
\item Finite time  blow-up on both sides $\pm t>0$, 
\item Scattering to $0$ as $t\to\I$ and finite time  blow-up in $t<0$, 
\item Finite time  blow-up in $t>0$ and scattering to $0$ as $t\to-\I$, 
\item Trapped by $\pm Q$ for $t\to\I$ and scattering to $0$ as $t\to-\I$, 
\item Scattering to $0$ as $t\to\I$ and trapped by $\pm Q$ as $t\to-\I$, 
\item Trapped by $\pm Q$ for $t\to\I$ and finite time  blow-up in $t<0$, 
\item Finite time  blow-up in $t>0$ and trapped by $\pm Q$ as $t\to-\I$, 
\item Trapped by $\pm Q$  as $t\to\pm\I$, 
\end{enumerate}
where ``trapped by $\pm Q$" means that the solution stays in a $O(\e)$ neighborhood of $\pm Q$ forever after some time (or before some time). The initial data sets for (1)-(4), respectively, are open. 
\end{thm}

The striking difference from the Kenig-Merle or Duyckaerts-Merle type results is the existence of solutions which blow up for $t<0$ and scatter for $t\to+\I$, or vice versa. It also implies that the initial data set for the forward scattering $(1)\cup(3)\cup(6)$ (or backward scattering) is unbounded in~$\HH$; in fact, it contains a curve connecting zero to infinity in~$\HH$. 
The number ``nine" simply means that all possible combinations of scattering to zero/scattering to $\pm Q$/finite time blow-up are allowed as $t\to\pm\I$. 
Each of these are in fact realized by infinitely many solutions. 

The simple dichotomy in terms of $K_0$ is no longer available, so it is not easy to predict the global dynamics for a given initial data, as in the case of Payne-Sattinger or Kenig-Merle below the ground state. However, we can give some description, in a more dynamical way, combining the hyperbolic structure around the ground state with the global variational structure. Let $L_+$ denote the linearized operator around $Q$, and let $\r$ be its ground state:
\EQ{ \label{def L+}
 L_+ = -\De + 1 - 3Q^2, \pq L_+\r = -k^2\r,\pq k>0,\ \r(x)>0,\ \|\r\|_{L^2}=1.}
It is well known that $L_+$ has only one negative eigenvalue. 
We can define a nonlinear distance function $d_Q(\vec u):\HH^\e\to[0,\I)$ continuous such that 
\EQ{
 \pt d_Q(\vec u) \simeq \inf_\pm\|\vec u\mp\vec Q\|_\HH,
 \pr d_Q(\vec u)\ll 1 \implies d_Q(\vec u)=E(\vec u)-E(Q)+k^2|\LR{u\mp Q|\r}|^2,}
where $\pm Q$ is chosen to be the closest to $u$. Let  
\EQ{ \label{def BQ}
 B(\pm Q):=\{\vec u\in \HH^\e \mid d_Q(\vec u)\le 2[E(\vec u)-E(Q)]\}}
be a pair of small balls around $\pm Q$. Outside of these balls, we can define a continuous sign function $\Sg:\HH^\e\setminus B(\pm Q)\to\{\pm 1\}$ such that 
\EQ{
 \pt d_Q(\vec u)\le\de_X \implies \Sg(\vec u)=-\sign\LR{u\mp Q|\r}, 
 \pr d_Q(\vec u)\ge \de_S \implies \Sg(\vec u)=\sign K_0(u),}
for some $\de_X>\de_S>2\e$, with the convention that $\sign 0=+1$. We emphasize that $\Sg$ is different\footnote{$\Sg$ might be the same as $\sign K_0$ in $\HH^\e\setminus B(\pm Q)$, but it depends on the coefficient $2$ in the definition of $B(\pm Q)$. In fact we can make it closer to $1$ as $d_Q(\vec u)\to 0$, and in that region we can show that $-\sign\LR{u\mp Q|\r}$ and $\sign K_0(u)$ are indeed opposite at some points.} 
 from $\sign K_0$ in the region close to $\pm Q$. 

\begin{thm} \label{Lachesis}
There are small $\e>0$ and $R>2\e$ with the following property. If $u$ is a solution of NLKG~\eqref{eq:NLKG} in $\HH^\e$, defined on an interval $I$, satisfying 
\EQ{ \label{exiting}
  d_Q(\vec u(\t_1))<d_Q(\vec u(\t_2))< R, \pq \vec u(\t_2)\not\in B(\pm Q),}
for some $\t_1<\t_2\in I$, then we have $d_Q(\vec u(t))>d_Q(\vec u(\t_2))$ for all $t>\t_2$ in $I$. In particular, $u$ remains outside of $B(\pm Q)$ with a fixed sign $\Sg(\vec u(t))\in\{\pm 1\}$ for $t>\t_2$. If $\Sg(\vec u(t))=+1$, then $u$ scatters to $0$ as $t\to\I$. If $\Sg(\vec u(t))=-1$, then $u$ blows up in
finite time after $\t_2$. 
Conversely, if a solution $u$ in $\HH^\e$ scatters as $t\to\I$, then $d_Q(\vec u(t))>R$, $\vec u(t)\not\in B(\pm Q)$ and $\Sg(\vec u(t))=+1$ for large $t$.
 If it blows up as $t\to T-0$, then $d_Q(\vec u(t))>R$, $\vec u(t)\not\in B(\pm Q)$ and $\Sg(\vec u(t))=-1$ for $t<T$ close to~$T$.
 If it is trapped by $\pm Q$ as $t\to\I$, then it stays in $B(\pm Q)$ for large $t$. 
\end{thm}
In other words, every solution can enter and exit $B(\pm Q)$ at most once, and the sign of $\Sg(\vec u)$ is constant while it is away from $B(\pm Q)$. 
Hence, if a solution does enter $B(\pm Q)$ and exits, then its fate is determined by the sign of $\Sg$ at the entrance and the exit, respectively, for $t\to-\I$ and $t\to+\I$, which are not necessarily the same (cf.~cases (3) and (4) of Theorem~\ref{thm:main}). It also implies that the set (9) in the previous theorem is bounded in the energy space $\HH$. 

In fact, we can give a more precise description of the exiting dynamics. Indeed, in the setting \eqref{exiting}, $d_Q(\vec u(t))$ is monotonically and exponentially growing for $t\ge\t_2$ until it reaches a larger number $\de_X$ (see Lemma \ref{2nd exit}). 

One can also give a more detailed description of the dynamics of those solutions which are trapped by $\pm Q$, by means of the following {\em spectral gap property} of the linearized operator $L_{+}=-\Delta+1-3Q^{2}$:
\begin{equation}\label{eq:gap}
 \text{$L_{+}$ has no eigenvalues in $(0,1]$ and no resonance}\footnote{In our radial case, it is easy to preclude by means of analytical arguments eigenvalues
  at the threshold $1$, but we include it here, because such an eigenvalue might exist in the nonradial case, which is indeed treated by \cite{DS}.} \text{ at the threshold $1$}
\end{equation}
It was verified by means of the Birman-Schwinger theorem and a numerical computation of the five largest eigenvalues of a suitably discretized Birman-Schwinger operator in~\cite{DS} by Demanet and the second author. 
Using this gap property one obtains the  following refinement of Theorem~\ref{thm:main}. 
 
\begin{thm}\label{thm:main*}
In the statement of Theorem~\ref{thm:main} one can replace ``trapped by $\pm Q$" with ``scattering to $\pm Q$". 
The sets $(5)\cup(7)\cup(9)$ and $(6)\cup(8)\cup(9)$ are smooth codimension one manifolds in the (radial) phase space $\HH$, and they are the\footnote{Since center manifolds are in general not unique it might be more precise to say ``a center-stable manifold'' here, but we ignore this issue. In fact, our manifolds are naturally unique for the global characterization in Theorem \ref{thm:main}.} center-stable manifold, resp.~the center-unstable manifold, around $\pm Q$. 
Similarly, (9) is a smooth manifold of codimension $2$, namely the center manifold. 
\end{thm}

It seems natural to the authors to separate the above theorem from the previous ones, as it is know from~\cite{DS} that~\eqref{eq:gap} fails if one lowers the power $3$ on the nonlinearity slightly, say to power $<2.8$. 
On the other hand, our argument for Theorems \ref{thm:main} and \ref{Lachesis} is quite general, in particular it does not use \eqref{eq:gap}, or more precisely, any dispersive property of the linearized operator. 
In fact, it is straightforward to extend them to all $L^2$ super-critical and $H^1$ subcritical powers and all space dimensions, i.e.
\EQ{
 \ddot u - \De u + u = u^p, \pq 1+4/d<p<1+4/(d-2), \pq u:\R^{1+d}\to\R.} 
 
Finally, from the above theorems, we can easily deduce a Duyckaerts-Merle-type result on the energy threshold.

\begin{thm} \label{thm:threshold}
Consider the limiting case $\e\to 0$ in Theorem \ref{thm:main}, i.e., all the radial solutions satisfying $E(\vec u) \le E(Q)$. 
Then the sets $(3)$ and $(4)$ vanish, while the sets $(5)-(9)$ are characterized, with some special solutions $W_\pm$, as follows
\EQ{
 \pt (5)=\{\pm W_-(t-t_0) \mid t_0\in\R\},
 \pq (6)=\{\pm W_-(-t-t_0) \mid t_0\in\R\},
 \pr (7)=\{\pm W_+(t-t_0) \mid t_0\in\R\},
 \pq (8)=\{\pm W_+(-t-t_0) \mid t_0\in\R\},
 \pr (9)=\{\pm Q(t-t_0) \mid t_0\in\R\}.}
 The solutions $W_{\pm}(t)$ converge exponentially to $Q$ as $t\to\I$. 
\end{thm}

As mentioned before $(5)\cup(7)\cup(9)$ is the stable manifold, and $(6)\cup(8)\cup(9)$  the unstable manifold, associated with $\pm Q$, cf.~\cite{BJ2}. 
The rest of the paper is organized as follows. In Section~2, we recall several known facts about the ground state and the linearized operator around it. In Section~3  we construct the center-stable manifold around the ground states, under the aforementioned spectral condition~\eqref{eq:gap}. Next to Theorem~\ref{thm:main*} this is  the {\em only part} where we require the gap property. Section~4 is the most important and novel part of this paper, where we prove a part of Theorem \ref{Lachesis} that every solution can enter and exit a small neighborhood of the ground states $\pm Q$ at most one time. We call it {\it ``one-pass theorem"}. One then immediately obtains the blow-up part of the dynamical classification by the classical Payne, Sattinger argument. In Section~5, we then prove the scattering part, by complementing the Kenig, Merle argument with the one-pass theorem. Finally in Section~6, we describe the global dynamics and its classification as in Theorem~\ref{thm:main} and Theorem~\ref{thm:main*}, using some simple topological arguments and the one-pass theorem. In the appendix we give a table of notation for frequently used symbols. 

\section{The ground state}
Here we recall several known properties of the ground state $Q$. 
First we consider its variational character with respect to the scaling symmetry. 
For our purpose, it suffices to consider the $L^0$ and the $L^2$ invariant scalings:
\EQ{\nn 
 \fy(x) \mapsto \fy^\nu_0(x):=\nu\fy(x), \pq \fy(x) \mapsto \fy^\nu_2(x):=\nu^{3/2}\fy(\nu x).}
Let $D_0$ and $D_2$, respectively,  be the generators of these symmetries, viz. 
\EQ{\nn 
 D_0\fy(x):=\fy(x)=\p_\nu|_{\nu=1}\fy_0^\nu, \pq D_2\fy(x):=\frac{x\na+\na x}{2}\fy(x)=\p_\nu|_{\nu=1}\fy_2^\nu,}
and let $K_0$ and $K_2$ be the derivatives of $J$  with respect to these scalings: 
\EQ{\nn \label{def K}
 \pt K_0(\fy):=\int_{\R^3} \big[|\na\fy|^2+|\fy|^2-|\fy|^4\big] \,dx=\p_\nu|_{\nu=1}J(\fy_0^\nu)=\LR{J'(\fy)|D_0\fy},
 \pr K_2(\fy):=\int_{\R^3} \big[ |\na\fy|^2-\frac{3}{4}|\fy|^4\big] \,dx=\p_\nu|_{\nu=1}J(\fy_2^\nu)=\LR{J'(\fy)|D_2\fy},}
where $J'(\fy)$ denotes the Fr\'echet derivative 
\EQ{\nn 
 J'(\fy)=-\De \fy + \fy - \fy^3.}
Define positive functionals $G_s$ for $s=0,2$ by 
\EQ{ \label{def G}
 G_0(\fy):=J-\frac{K_0}{4}=\frac 14\|\fy\|_{H^1}^2,
 \pq G_2(\fy):=J-\frac{K_2}{3}=\frac 16\|\na\fy\|_{L^2}^2+\frac 12\|\fy\|_{L^2}^2.}

\begin{lem} \label{minimization}
For $s=0,2$ we have 
\EQ{ \label{cnstr min}
J(Q) \pt= \inf\{J(\fy) \mid 0\not=\fy\in H^1,\ K_s(\fy)=0\}   
 \pr= \inf\{G_s(\fy) \mid 0\not=\fy\in H^1,\ K_s(\fy)\le 0\},
}
and these infima are achieved uniquely by the ground states $\pm Q$.
\end{lem}
\begin{proof}
If $K_s(\fy)<0$ then $K_s(\la_*\fy)=0$ for some $\la_{*}\in(0,1)$, whereas $G_s(\la_{*}\fy)<G_s(\fy)$. 
Since moreover $K_s(\fy)=0$ implies $J(\fy)=G_s(\fy)$, the two infima are equal.   Since 
 $J'(Q)=0$ implies $K_s(Q)=0$, the infima are no larger than $J(Q)$.

To obtain a minimizer, let $\{\fy_n\}_{n\ge1}\subset H^1_{\mathrm{rad}}\setminus\{0\}$ be a minimizing sequence such that (the Schwartz symmetrization allows us to restrict them to the radial functions)
\EQ{\nn 
 K_s(\fy_n)=0, \pq J(\fy_n)\to m,}
where $m$ denotes the right hand side of~\eqref{cnstr min}. 
Since $G_s(\fy_n)=J(\fy_n)$ is bounded, $\fy_n$ is bounded in $H^1$. After extraction of a subsequence, it converges weakly to some $\fy_\infty$ in $H^1$, and in the strong sense in~$L^4$, by the radial symmetry. 
Thus $K_s(\fy_\infty)\le 0$, $J(\fy_\I)\le J(Q)$ and $G_s(\fy_\infty)\le G_s(Q)$. If $\fy_\I=0$ then the strong convergence in $L^4$ together with $K_s(\fy_n)\to 0$ implies that $\fy_n\to 0$ strongly in $H^1$. Notice that $K_s(\fy)>0$ for $0<\|\fy\|_{H^1}\ll 1$, due to the interpolation inequality
\EQ{ \label{GN}
 \|\fy\|_{L^4}^4 \lec \|\fy\|_{L^2}\|\na\fy\|_{L^2}^3.} 
Hence $K_s(\fy_n)>0$ for large $n$, a contradiction. Thus we obtain a nonzero minimizer $\fy_\I$, and so $K_s(\fy_\I)=0$, which implies that $\fy_n\to\fy_\I$ strongly in $H^1$. 

The constrained minimization implies that for some Lagrange multiplier $\mu\in\R$, 
\EQ{\nn 
 J'(\fy_\I)=\mu K_s'(\fy_\I), \pq K_s(\fy_\I)=0,}
Multiplying the Euler-Lagrange equation with $D_s\fy$ and integrating by parts yields
\EQ{\nn
 0=K_s(\fy_\I)=\mu\LR{K_s'(\fy_\I)|D_s(\fy_\I)}
 =\mu\times\CAS{ 2K_0(\fy_\I)-2\|\fy_\I\|_{L^4}^4&(s=0), \\
  2K_2(\fy_\I)-\frac{3}{2}\|\fy_\I\|_{L^4}^4&(s=2),}}
which implies $\mu=0$, so $\fy_\I$ is a ground state. 
\end{proof}

Next, we record the following observation from~\cite{IMN}. 

\begin{lem} \label{sign K below Q}
For any $\fy\in H^1$ such that $J(\fy)<J(Q)$, one has either $\fy=0$, or $K_s(\fy)>0$ for both $s=0,2$, or $K_s(\fy)<0$ for both $s=0,2$. 
\end{lem}
\begin{proof}
In the proof of the above lemma, we already know that $0$ is surrounded by the open set $K_s(\fy)>0$ for both $s$, and $K_s(\fy)=0$ is prohibited except for $0$ by $J(\fy)<J(Q)$, due to the above minimization property of $Q$. Then it is enough to show that each set $\K_s^+:=\{K_s\ge 0, J<J(Q)\}$ is contractible to $\{0\}$, which implies that it is connected and so cannot be divided by the disjoint open sets $\K_{2-s}^+$ and $\K_{2-s}^-=\{K_{2-s}<0, J<J(Q)\}$. 
The contraction is given by $\fy^\nu_s$ with $\nu:1\to 0$. By definition $\p_\nu J(\fy^\nu_s)=K_s(\fy_s^\nu)/\nu$, and so $J(\fy^\nu_s)$ decreases together with $\nu$, as long as $K_s(\fy_s^\nu)>0$, which is preserved as long as $J(\fy_s^\nu)<J(Q)$. Hence, $\fy_s^\nu$ stays in $\K_s^+$ for $1\ge\nu>0$. When $s=2$, it converges to $0$ as $\nu\to+0$ only in $\dot H^1\cap L^4$, but since $\K_2^\pm$ are open in that topology this is sufficient: once $\fy_2^\nu$ gets in a small ball around $0$ in $\K_2^+$, one can change to the other scaling $\fy\mapsto\nu\fy$ to send it to $0$ in $H^1$. 
\end{proof}

Next we recall the spectral properties of $Q$ and the linearized operator $L_+$ defined in \eqref{def L+}. Decomposing the solution of~\eqref{eq:NLKG} in the form
\EQ{\nn 
 u = Q + v,}
we obtain the equation of the remainder 
\EQ{ \label{eq:KGL}
 \pt \ddot v + L_+v = N(v), 
 \pq N(v):=(v+Q)^3-Q^3-3Q^2v=3Qv^2+v^3.}
The energy functionals are expanded correspondingly
\EQ{ \label{exp J}
 J(Q+v)\pt=J(Q)+\LR{J'(Q)|v}+\LR{J''(Q)v|v}/2+O(\|v\|_{H^1}^3)
 \pr=J(Q)+\LR{L_+v|v}/2+O(\|v\|_{H^1}^3),
 \\ K_s(Q+v)&=\LR{K_s'(Q)|v}+\LR{K_s''(Q)v|v}/2+O(\|v\|_{H^1}^3).}
In particular 
\EQ{ \label{exp K0}
 K_0(Q+v)\pt=2\LR{-\De Q+Q-2Q^3|v}+\LR{(-\De+1-6Q^2)v|v}+O(\|v\|_{H^1}^3)
 \pr=-2\LR{Q^3|v}+\LR{(L_+-3Q^2)v|v}+O(\|v\|_{H^1}^3).}

\begin{lem} \label{lem:spec} As an operator in $L^{2}_{\mathrm{rad}}$, 
$L_+$ has only one negative eigenvalue, which is non-degenerate, and no eigenvalue at $0$ or in the continuous spectrum $[1,\I)$. 
\end{lem}
\begin{proof}
$L_+$ has at least one negative eigenvalue because 
\EQ{\nn
 \langle L_+Q|Q \rangle =-3\|Q\|_4^4<0,}
and at most one, because 
\EQ{\nn
 \LR{Q^3|v}=0 \implies \LR{L_+v|v}\ge 0.}
To see this, suppose that $f\in H^1$ satisfies $\LR{Q^3|f}=0$ and $\LR{L_+f|f}=-1$, and let $v=\e Q+\de f$ for small $\e,\de\in\R$. Then 
\EQ{\nn
 K_0(Q+v)=-2\e\|Q\|_{L^4}^4 - \de^2(1+\LR{3Q^2f|f}) + O(\e^2+\e\de+\de^3),}
so there exists $\e=O(\de^2)$ such that $K_0(Q+v)=0$. On the other hand
\EQ{\nn
 J(Q+v)=J(Q)-\de^2+O(\de^3)<J(Q),}
which contradicts the minimizing property Lemma \ref{minimization}. 

Next, if $0\not=f\in L^{2}_{\mathrm{rad}}$ solves $L_+f=0$, then $f\perp Q^3,Q$, because $L_+Q=-2Q^3$ and $L_+(r\p_r+1)Q=-2Q$. 
By the Sturm-Liouville theory, $f$ can change the sign only once, say at $r=r_0>0$, so does $Q^3-Q(r_0)^2Q$, since $Q(r)$ is decreasing. Then $(Q^3-Q(r_0)^2Q)f$ is non-zero with a definite sign, contradicting $f\perp Q^3,Q$. 

The absence of embedded eigenvalue is standard, and follows also from the asymptotic equation for any eigenfunction, viz. 
\EQ{\nn 
 L_+f=\la^2 f \implies (\p_r^2+\la^2)(rf)=-3Q^2(rf)\lec e^{-2r},}
by the exponential decay of $Q$. 
\end{proof}

Let 
\EQ{ \label{def P+}
 P^+:=1-\r \langle \r|} 
be the orthogonal projection for $\r$. Then the above lemma implies $\LR{L_+v|P^+v}\gec\|v\|_{L^2}^2$, 
and so for any $v\in P^+(H^1)$ and $\theta\in(0,1]$ sufficiently small, 
\EQ{\nn 
 \LR{L_+v|v}\gec (1-\theta)\|v\|_{2}^{2} + \theta\|v\|_{H^1}^2 -3\theta\LR{Q^2v|v}\gec \|v\|_{H^1}^2,}
hence $\LR{L_+v|v}\simeq\|v\|_{H^1}^2$ on $P^+(H^1)$. 

The above property of $L_+$ is sufficient for the analysis of dynamics away from $Q$, but for that of solutions staying forever around the ground state, we require the following property of $L_+$: 
\EQ{ \label{hyp:spec}
 \text{$L_+$ has no eigenvalue in $(0,1]$ and no resonance at the threshold $1$.}}
Both parts of this statement have been verified by Demanet and the second author in~\cite{DS} via numerics. 
Their approach is based on the Birman-Schwinger theorem which equates the number of eigenvalues and resonances $\le1$ (the threshold) of $L_{+}$, counted with multiplicity, to the number of eigenvalues $\ge1$ of  the self-adjoint, positive, compact operator $K_{+}:=3Q(-\Delta)^{-1}Q$, again counted with multiplicity. 
In fact, \cite{DS} by means of a numerical computation finds that $K_{+}$ has precisely four eigenvalues greater than $1$ (corresponding to the negative ground state plus the zero eigenvalue of multiplicity three: $L_{+}\nabla Q=0$), whereas the fifth largest eigenvalue was calculated to be $<0.98$ with an estimated 8 to 9 digits of accuracy behind the comma (to be precise, $\lambda_{5}=0.97039244\ldots$). 
By the Birman-Schwinger theorem, this verifies~\eqref{hyp:spec}.

\section{Center-stable manifold} \label{s:mfd}

In this section, we investigate the solutions staying around the 
unstable ground state, constructing the center-stable manifold, while the center-unstable manifold is obtained by reversing the time. The situation is much simpler than for NLS~\cite{S}, because the only symmetry present (reflection and time translation) fixes the ground state and the linearized operator is scalar self-adjoint. 
We decompose any solution $u$ simply by putting 
\EQ{\nn 
 u(t) = Q+v(t), \pq v(t) = \la(t)\r + \ga, \pq \ga\perp\r,}
which is obviously unique. We then obtain the equations of $(\la,\ga)\in\R\times P^+(H^1)$
\EQ{ \label{eq:Lgamma}
 \CAS{ \ddot\la - k^2 \la = P_\r N(v) =: N_\r(v),\\
  \ddot\ga + \om^2 \ga = P^+ N(v) =: N_c(v), \pq \om:=\sqrt{P^+L_+}.}}

We look for a forward global solution which grows at most polynomially, to which it is equivalent to remove 
the growing mode $e^{kt}$. From the integral equation of $\la$ one extracts the growing mode 
\EQ{\nn 
 \la(t) \pt= \cosh (kt) \la(0) + \frac 1k\sinh (kt) \dot\la(0) + \frac 1k\int_0^t\sinh (k(t-s)) N_\r(v)(s)\, ds
 \pr= \frac{e^{kt}}{2}\left[\la(0)+\frac 1k\dot\la(0)+\frac 1k\int_0^t e^{-ks}N_\r(v)(s)\, ds\right]+\cdots}
 where  the omitted terms are exponentially decaying. 
Hence, the necessary and sufficient stability condition is 
\EQ{ \label{kill unstb}
 \dot\la(0)=-k\la(0) - \int_0^\I e^{-ks}N_\r(v)(s)\, ds.}
Under this condition, the integral equation for $\la(t)$ is reduced to 
\EQ{ \label{inteq la red}
 \la(t) = e^{-kt}\left[\la(0)+\frac{1}{2k}\int_0^\I e^{-ks} N_\r(v)(s)\, ds\right]+\frac{1}{2k}\int_0^\I e^{-k|t-s|}N_\r(v)(s)\, ds,}
while the integral equation for $\ga$ is 
\EQ{ \label{inteq ga}
 \ga(t) = \cos (\om t )\ga(0) + \frac 1\om \sin (\om t) \dot\ga(0) + \frac 1\om \int_0^t \sin( \om (t-s)) N_c(v)(s)\, ds.}
The coupled equations \eqref{inteq la red}--\eqref{inteq ga} can be solved by iteration, using the Strichartz estimate for $e^{it\om}$, for any small 
initial data $(\la(0),\vec\ga(0))\in\R\times P^+(\HH)$. 
The linearized energy norm of $v$ around $Q$ is denoted by 
\EQ{ \label{def Enorm}
 \|\vec v\|_{E} \pt:= \sqrt{[k^2\LR{v|\r}^2+\|\om P^+v\|_{L^2}^2+\|\dot v\|_{L^2}^2]/2}
 \pr=\sqrt{[k^2|\la|^2+|\dot \la|^2]/2+\|\ga\|_{E}^2}.}
By Lemma \ref{lem:spec} we have
\EQ{
  \|\vec v\|_E^2 \simeq \|\vec v\|_{\HH}^2 = \|v\|_{H^1}^2+\|\dot v\|_{L^2}^2.}
Recall that $\HH$ is the radial energy space. 

\begin{prop} \label{prop:mfd}
Assume that \eqref{hyp:spec} holds. 
Then there are $\nu>0$ and $C\ge 1$ with the following property: For any given $\la(0)\in\R$, $\vec\ga(0)\in P^+(\HH)$ satisfying 
\EQ{\nn
 E_0:=k^2|\la(0)|^2+\|\om\ga(0)\|_{L^2}^2+\|\dot\ga(0)\|_{L^2}^2 \le \nu^2,}
there exists a unique solution $u$ of NLKG \eqref{eq:NLKG} on $0\le t<\I$ satisfying 
\EQ{\nn 
 \pt u(0)=Q+\la(0)\r+\ga(0), \pq P^+\dot u(0)=\dot\ga(0),}
$|\LR{\r|\dot u(t)+k u(t)}|\lec E_0$ for all $t\ge 0$, and 
\EQ{ \label{u near Q}
 \|\vec u(t) - \vec Q\|_\HH^2 \le C E_0 \pq(0\le\forall t<\I).}
The dependence of $u$ on $(\la(0),\vec\ga(0))$ is smooth in $L^\I(0,\I;\HH)$. 
In addition, there exists a unique free Klein-Gordon solution $\ga_\I$ such that 
\EQ{\nn 
 |\la(t)|+|\dot \la(t)|+\|\vec\ga(t)-\vec\ga_\I(t)\|_\HH \to 0,} 
as $t\to\I$. In particular, we have $E(\vec u)=J(Q)+\|\vec\ga_\I\|_{\HH}^2/2$. 

Conversely, any solution $u$ of NLKG satisfying \eqref{u near Q} with $E_0\le \nu^2/C$ must be given in this way, which is uniquely determined by $(\la(0),\vec\ga(0))$, and by $(\la(0),\vec \ga_\I(0))$. 
\end{prop}

An alternate, but essentially equivalent, formulation of this result reads as follows.

\begin{prop} \label{prop:mfd'}
Assume that \eqref{hyp:spec} holds. Then there exists $\nu>0$ small and a smooth graph~$\M$ in $B_{\nu}(Q)\subset\HH$ so that $\M$ is tangent to 
\EQ{\label{eq:TQM}
T_{Q}\M=\{(u_{0},u_{1})\in\HH\mid \langle k u_{0} + u_{1}|\rho\rangle =0\}
}
in the sense that 
\[
\sup_{x\in\partial B_{\delta}(Q)} \dist(x , T_{Q}\M)\lec \delta^{2} \quad \forall\; 0<\delta<\nu
\]
and so that any data $(u_{0},u_{1})\in \M$ lead to global evolutions of~\eqref{eq:NLKG} of the form $u=Q+v$
where $v$ scatters to a free Klein-Gordon solution in~$\HH$. 
Moreover, no solution can stay off $\M$ and inside $B_{\nu}(Q)$ for all $t>0$, and, $\M$ is invariant under the flow for all $t\ge0$. 
\end{prop}

\begin{proof}[Proof of Proposition~\ref{prop:mfd}]
For the existence, we solve \eqref{inteq la red}--\eqref{inteq ga} by iteration using the norm 
\EQ{\nn 
 \pt\|(\la,\ga)\|_X := \|\la\|_{L^1\cap L^\I(0,\I)} + \|\ga\|_{St(0,\I)}, 
 \pq St:=L^2_t L^6_x \cap L^\infty_t H^1_x. }
The Strichartz estimate for the free Klein-Gordon equation gives us 
\EQ{ \label{eq:strichKG1}
 \|u\|_{St(0,T)} \lec \|\vec u(0)\|_{\HH} + \|(\p_t^2-\De+1)u\|_{L^1_tL^2_x(0,T)}.}
Under the hypothesis \eqref{hyp:spec}, the operator $L_+$ satisfies the conditions of Yajima's $W^{k,p}$ boundedness theorem for the wave operators \cite{Y}, so that we can conclude that~\eqref{eq:strichKG1} applies to $\gamma$ as well:
\EQ{\nn 
 \|\ga\|_{St(0,T)} \lec \|\vec \ga(0)\|_\HH+\|(\p_t^2-\om^2)\ga\|_{L^1_tL^2_x(0,T)},}
provided that $\vec\ga(0)$ and $(\p_t^2-\om^2)\ga$ are orthogonal to $\r$. 

For the solution $\la$ of \eqref{inteq la red}, with $\dot\lambda(0)$ uniquely determined by~\eqref{kill unstb},
we can estimate the norm by simple integration in $t$: 
\EQ{\nn 
 \|\la\|_{L^1\cap L^\I(0,\I)} \lec k^{-1}|\la(0)| + k^{-1}\|N_\r(v)\|_{L^1_t(0,\I)}
 \lec |\la(0)|+\|N(v)\|_{L^1_tL^2_x(0,\I)},}
and for the solution $\ga$ of \eqref{inteq ga} by the above Strichartz estimate 
\EQ{\nn 
 \|\ga\|_{St(0,\I)} \lec \|\vec \ga(0)\|_\HH+\|N(v)\|_{L^1_tL^2_x(0,\I)}.}
The nonlinearity $N(v)$ is bounded in $L^1_tL^2_x$ by 
\EQ{\label{eq:g3}
 \|Q v^2\|_{L^1_tL^2_x} + \|v^3 \|_{L^1_tL^2_x}  \lec \|v\|_{L^2_tL^6_x}^2(\|v\|_{L^\I_tL^6_x}+\|Q\|_{L^\I_tL^6_x}),}
where the norm of $v$ is bounded by 
\EQ{\nn 
 \pt\|v\|_{L^p_tL^6_x} \lec \|\la\|_{L^1_t\cap L^\I_t} + \|\ga\|_{L^p_tL^6_x} \pq (1\le p\le\I).}
Gathering them, we obtain 
\EQ{\nn 
 \|(\la,\ga)\|_X \lec |\la(0)|+\|\vec\ga(0)\|_\HH+\|(\la,\ga)\|_X^2+\|(\la,\ga)\|_X^3.}
Applying these estimates to the iteration sequence, we obtain a unique fixed point of \eqref{inteq la red}--\eqref{inteq ga} for any given small $(\la(0),\vec\ga(0))$. It is straightforward to see that $u:=Q+\la(t)\r+\ga(t)$ solves NLKG on $0\le t<\I$, satisfying 
\EQ{\nn 
 \|\vec u - \vec Q\|_\HH \lec \|(\la,\ga)\|_X \lec |\la(0)|+\|\vec\ga(0)\|_\HH,}with smooth dependence on the data. The bound on $\LR{\r|\dot u+ku}=\dot\la+k\la$ follows by using the equation once again. 

Moreover, the asymptotic profile of $\ga$ is given by 
\EQ{\nn 
 \ga_\I(t)=\cos (\om t )\ga(0) + \frac 1\om \sin (\om t)\dot\ga(0) + \frac 1\om \int_0^\I \sin( \om (t-s)) N_c(v)(s)\, ds,}
with  the convergence property
\EQ{\nn 
 \|\vec\la\|_{L^\I_t(T,\I)}+\|\vec\ga-\vec\ga_\I\|_{L^\I_t\HH(T,\I)}\lec \|N(v)\|_{L^1_tL^2(T,\I)}\to 0 \pq(T\to\I).}
The iteration can be solved with a given $\ga_\I$ and the equation of $\ga$ now reads 
\EQ{ \label{gaeq infty}
 \ga(t)=\ga_\I(t)+\frac 1\om \int_\I^t \sin( \om (t-s)) N_c(v)(s)\, ds,}
where the estimates are essentially the same. $\ga_\I(t)$ can be further replaced with a free Klein-Gordon solution by the linear scattering for $L_+$. 

The uniqueness part requires some more work, since a priori we do not know if the solution is in the space $X$, globally in time. 
Let $u$ be a solution on $[0,\I)$ satisfying \eqref{u near Q}. Since it is bounded in the energy space, we can easily see that $N_\r(t)$ is bounded. Therefore, it has to satisfy \eqref{kill unstb}, and the reduced integral equation \eqref{inteq la red} as well as \eqref{inteq ga} for all $0<t<\I$. 
To see that $\la\in L^1_t(0,\I)$ and $\ga\in L^2_tL^6_x(0,\I)$, consider the norm 
\EQ{ \label{semifinite norm}
 \|(\la,\ga)\|_{X_T}:=\|\la\|_{L^1(0,T)} + \|\la\|_{L^\I(T,\I)} + \|\ga\|_{St(0,T)},}
for $T>0$. The energy bound implies that $\|(\la,\ga)\|_{X_T}<\I$ for all $T>0$, but we require a uniform bound. 
From the integral equations one concludes that 
\EQ{\nn 
 \|\la\|_{L^1(0,T)\cap L^\I(T,\I)} \pt\lec \|N_\r\|_{L^1(0,T)\cap L^\I(T,\I)} 
 \pr\lec \nu + \nu(\|\la\|_{L^1(0,T)\cap L^\I(T,\I)}+\|\ga\|_{St(0,T)}),} 
and using the Strichartz estimate, one further has
\EQ{\nn 
 \|\ga\|_{St(0,T)} \lec \nu + \nu(\|\la\|_{L^1(0,T)}+\|\ga\|_{St(0,T)}),}
Thus we obtain by Fatou
\EQ{\nn 
 \|(\la,\ga)\|_X \le \liminf_{T\to\I} \|(\la,\ga)\|_{X_T} \lec \nu,} 
and the contraction mapping principle implies the uniqueness. 
\end{proof}

As for Proposition~\ref{prop:mfd'}, we only need to let $\M$ be those $\vec u(0)$ for all the solutions $u$ constructed above, satisfying \eqref{u near Q} with $E_0=\nu^2/C$. 

\section{One-pass theorem}\label{sec:no homo}
 
The key observation in our analysis of global dynamics is that any solution with energy only slightly higher than the ground states 
{\it can come close to the ground states at most once}. More precisely, if a solution $u$ passes in and out of a small neighborhood of $\{\pm Q\}$, then it can never come back again. In particular, there is no homoclinic orbit connecting $\pm Q$ with themselves. Even though an orbit connecting $Q$ and $-Q$ should be called heteroclinic, we are regarding it as homoclinic, by identifying $\pm Q$ as one point, since there will be no difference in precluding them by our argument. This point of view becomes more natural if one considers NLS, where the ground state is really a connected set (topologically a cylinder in the radial case) generated from $Q$ by the invariant group action (the modulation and the scaling symmetries). 

The key ideas to preclude homoclinic orbits are:
\begin{itemize}
\item When a solution comes very close to $Q$, but does not fall on the center-stable\footnote{If one wishes not make  any reference to Proposition~\ref{prop:mfd}, then the dichotomy expressed in this idea simply becomes the general ``trapped by~$Q$'' or ``non-trapped by $Q$'' distinction, cf.~Theorem~\ref{thm:main}.}
 manifold, then it eventually 
leaves any small neighborhood of $Q$ in such a way that the unstable, i.e., exponentially growing, mode dominates all the others (the stable
and dispersive\footnote{However, we do not control quantitatively what ``eventually'' means, nor do we need to. The dynamics that takes place before
the exponential expansion dominates is very complicated and relies on an interplay between the different components, and the dispersive PDE
behavior can be of the same order of magnitude as the other dynamics. We therefore treat the ``pre-exit'' dynamics as a black box.} 
ones). This also applies to the 
negative time direction. 
\item For the solutions ejected from a small neighborhood of $Q$ with a dominating velocity in the unstable direction, we can use the virial identity, after suitable localization, as a Lyapunov-type quantity. 
\end{itemize}
The localized virial identity will be used in the following form. For a smooth function $w$ cutting-off the outside of light cones, we have 
\EQ{ \label{loc virial*}
 V_w(t):=\LR{wu_t|(x\na+\na x)u}, \pq \dot V_w(t)=-K_2(u(t)) + error,}
where the error term is due to the cut-off and bounded by the linear energy at $t$ in the exterior region. 

In order to use $V_w$ as a Lyapunov functional, we need a lower bound on $|K_2(u)|$, or $|K_0(u)|$ if we use the convexity of $L^2$ norm instead. 
In Payne-Sattinger \cite{PS} and Kenig-Merle \cite{KM1,KM2}, this is achieved solely by the variational structure, which is not sufficient by itself in our setting. Hence, our lower bound comes in two ways: 
\begin{itemize}
\item While $\vec u(t)$ is away from but still close to $\pm Q$, we use the hyperbolic nature of the eigenmode $\la$ (Lemma \ref{2nd exit}). 
\item While $\vec u(t)$ is not so close or really far away from $\pm Q$, we use the variational structure (Lemma \ref{K lower bd}). 
\end{itemize}
Note that the bound by $\la$ is not meaningful by itself once $\vec u(t)$ moves far away from $\pm Q$, since it is based on the linearization. The variational estimate is not useful close to $\pm Q$, even if we knew that $\vec u(t)$ does not really approach $\pm Q$, because that lower bound depends badly on the distance from $\pm Q$. However, for small $\e>0$ these two estimates exhibit sufficient overlap of their regions of validity. 

Now we introduce a nonlinear distance function to $\pm Q$, which seems best suited in order to exploit the hyperbolic dynamics together with the nonlinear energy structure. 
Let 
\EQ{ \label{decop u}
 u = \si[Q+v], \pq v=\la\r+\ga, \pq \ga\perp\r}
for $\si=\pm$, and decompose the energy into the linearized part~\eqref{def Enorm} and the higher order:
\EQ{ \label{energy exp}
 \pt E(\vec u)-J(Q)+k^2\la^2=\|\vec v\|_E^2 - C(v), 
 \pq C(v):=\LR{Q|v^3}+\|v\|_{L^4}^4/4.}
There exists $0<\de_E\ll 1$ such that 
\EQ{ \label{def dE}
 \|\vec v\|_E\le 4\de_E \implies |C(v)|\le\|\vec v\|_E^2/2.}
Let $\chi$ be a smooth function on $\R$ such that $\chi(r)=1$ for $|r|\le 1$ and $\chi(r)=0$ for $|r|\ge 2$. We define 
\EQ{
 d_\si(\vec u):= \sqrt{\|\vec v\|_E^2 - \chi(\|\vec v\|_E/(2\de_E))C(v)}.}
It has the following properties
\EQ{
  \|\vec v\|_E/2 \le d_\si(\vec u) \le 2\|\vec v\|_E,\pq d_\si(\vec u)=\|\vec v\|_E+O(\|\vec v\|_E^2),}
\EQ{ \label{energy dist}
 d_\si(\vec u)\le \de_E \implies d_\si(\vec u)^2=E(\vec u)-J(Q)+k^2\la^2.}

Henceforth, we shall always assume that $\vec u$ is decomposed as in \eqref{decop u} such that 
\EQ{\nn \label{def dQ}
 d_Q(\vec u):=\inf_\pm d_\pm(\vec u) = d_\si(\vec u),}
where the choice of sign $\si$ is unique as long as $d_Q(\vec u)\le 2\de_E$. We also set 
\EQ{\nn \label{def la+-}
 \la_\pm(t) := \la(t)\pm \dot\la(t)/k,}
the unstable/stable modes for $t\to\I$ relative to the linearized hyperbolic evolution, see~\eqref{kill unstb}. 
First, we investigate the solutions which are close to $\pm Q$ but which are moving away. 

\subsection{Eigenmode dominance} 
The first observation is that the eigenmode $\la$ becomes dominant in the energy and has a fixed sign, once $\vec u$ is slightly away from the ground state, compared with the energy level. This is a static statement in the phase space $\HH$, and an immediate consequence of the definition of $d_Q$. 
\begin{lem} \label{1st exit}
For any $\vec u\in\HH$ satisfying  
\EQ{
 E(\vec u)<J(Q)+d_Q(\vec u)^2/2, \pq d_Q(\vec u)\le \de_E,}
one has $d_Q(\vec u)\simeq|\la|$.  
\end{lem}
In particular, $\la$ has a fixed sign in each connected component of the above region.
\begin{proof}
\eqref{energy dist} yields
\EQ{\nn 
 d_Q(\vec u)^2=E(\vec u)-J(Q)+k^2\la^2<d_Q(\vec u)^2/2+k^2\la^2.}
and so, $k^2\la^2/4\le\|\vec v\|_E^2/2\le d_Q(\vec u)^2/2<k^2\la^2$. 
\end{proof}

\subsection{Ejection process} \label{ss:eject}

The following lemma is the key to extract the hyperbolic nature from our PDE. 
Here the linearized evolution of the eigenmode $\la$ plays the main role, and we 
specify that the solution is exiting rather than entering, by \eqref{exiting condition}. 

\begin{lem} \label{2nd exit}
There exists a constant $0<\de_X\le \de_E$ with the following property. Let $u(t)$ be a local solution of NLKG on an interval $[0,T]$ satisfying 
\EQ{\nn 
 R:=d_Q(\vec u(0)) \le \de_X, \pq E(\vec u)<J(Q)+R^2/2} 
and for some $t_0\in(0,T)$, 
\EQ{ \label{exiting condition}
 d_Q(\vec u(t)) \ge R \pq (0<\forall t<t_0).}
Then $d_Q(\vec u(t))$ increases monotonically until reaching $\de_X$, and meanwhile, 
\EQ{\nn 
  \pt d_Q(\vec u(t)) \simeq -\sg\la(t) \simeq -\sg\la_+(t) \simeq e^{kt}R, 
  \pr |\la_-(t)|+\|\vec \ga(t)\|_E \lec R+e^{2kt}R^2, 
  \pr \min_{s=0,2}\sg K_s(u(t)) \gec d_Q(\vec u(t)) - C_*d_Q(\vec u(0)),}
for either $\sg=+1$ or $\sg=-1$, where $C_*\ge 1$ is a constant. 
\end{lem}
\begin{proof}
Lemma \ref{1st exit} yields $d_Q(\vec u)\simeq|\la|$ as long as $R\le d_Q(\vec u)\le \de_E$, 
whereas the energy conservation of NLKG and the equation of $\la$ give as long as $d_Q(\vec u)\le \de_E$,  see~\eqref{def Enorm}, 
\EQ{
 \pt \p_t d_Q(\vec u)^2=2k^2\la\dot\la, 
 \pq \p_t^2 d_Q(\vec u)^2=2k^2|\dot\la|^2+2k^4|\la|^2+2k^2\la N_\r(v).}
The exiting condition \eqref{exiting condition} implies $\p_td_Q(\vec u)^2|_{t=0}\ge 0$. 
Since $N_\r(v)\lec\|v\|_{H^1}^2$, we have $\p_t^2d_Q(\vec u)^2\simeq d_Q(\vec u)^2$ as long as $d_Q(\vec u)\simeq|\la|\ll 1$. 

Hence, imposing $\de_X\le \de_E$ and small enough, we deduce that 
$d_Q(\vec u)\ge R$ strictly increases until it reaches $\de_X$; meanwhile, $d_Q(\vec u)\simeq\sg\la$ for $\sg\in\{\pm 1\}$ fixed. 
Since 
\EQ{
 \la_+^2-\la_-^2=4\la\dot\la/k\ge 0,} 
we also infer that $\la_+\simeq\la$. 

Next, integrating the equation \eqref{eq:Lgamma} for $\la$ yields
\EQ{
 |\vec\la(t)-\vec\la_0(t)| \lec \int_0^t e^{k(t-s)}|N_\r(v(s))|ds \lec \int_0^t e^{k(t-s)}|\la(s)|^2ds,}
where $\la_0$ denotes the linearized solution, which satisfies 
\EQ{
 \pt |\la_0(t)|=|e^{kt}\la_+(0)+e^{-kt}\la_-(0)|/2 \simeq Re^{kt}.}
Then by continuity in time we deduce 
\EQ{
 \la(t) \simeq -\sg Re^{kt}, \pq |\vec\la(t)-\vec\la_0(t)| \lec R^2e^{2kt},}
as long as $Re^{kt}\ll 1$. This yields the upper bounds on $\la_\pm$ as well. 

To bound the remainder $\ga$, we use the energy for the $\la$ equation, see \eqref{energy exp}, 
\EQ{
 |\p_t[-k^2\la/2+\dot\la^2/2-C(\la\r)]|=|(N_\r(v)-N_\r(\la\r))\dot\la|
 \lec \|\ga\|_{H^1}|\la|^2.}
Subtracting it from the energy \eqref{energy exp} yields
\EQ{
 |\p_t[\|\vec\ga\|_E^2-C(v)+C(\la\r)]| \lec \|\ga\|_{H^1}|\la|^2.}
Integrating this bound and using the bound on $\la$ and $\vec\ga(0)$, one obtains
\EQ{
 \|\vec\ga\|_{L^\I_tE(0,T)}^2 \lec R^2+\|\vec\ga\|_{L^\I_tE(0,T)}R^2e^{2kT},}
which implies the desired bound on $\ga$.  

Finally, recall from \eqref{exp K0}  that
\EQ{
 K_0(u)=-k^2\la\LR{Q|\r}-\LR{2Q^3|\ga}+O(\|v\|_{H^1}^2),}
and similarly we can expand $K_2$ around $Q$:
\EQ{ \label{exp K2}
 K_2(u)
 =-(k^2/2+2)\la\LR{Q|\r}-\LR{2Q+Q^3|\ga}+O(\|v\|_{H^1}^2).}
Since $\LR{Q|\r}>0$ by their positivity, we obtain the desired bound on $K_s$. 
\end{proof}
Note that the above proof did not use the equation for $\ga$, see~\eqref{eq:Lgamma}. 
Although we do have the full Strichartz estimate for the linearized evolution of $\ga$ at our disposal,
 thanks to the spectral gap condition \eqref{eq:gap}, the above proof does not require 
 any dispersive nature of $\ga$, and so is applicable to more general cases. 
This is indeed natural since we are dealing with that part of dynamics dominated by the hyperbolicity in~$\la$. 

\subsection{Variational lower bounds}
Now we turn to the solutions away from $\pm Q$. The variational estimate is derived as an extension of Lemma~\ref{sign K below Q}. This is essentially a static statement, where $\vec u=(u,\dot u)$ should be simply regarded as a point in the phase space $\HH$. 
\begin{lem} \label{K lower bd}
For any $\de>0$, there exist $\e_0(\de), \ka_0, \ka_1(\de)>0$ such that for any $\vec u\in\HH$ satisfying 
\EQ{ \label{energy region}
 E(\vec u)< J(Q)+\e_0(\de)^2, \pq d_Q(\vec u) \ge \de,}
one has either
\EQ{ \label{-K bd}
 K_0(u) \le -\ka_1(\de) \pq and \pq K_2(u) \le -\ka_1(\de),}
or 
\EQ{ \label{+K bd}
 K_0(u) \ge \min(\ka_1(\de), \ka_0\|u\|_{H^1}^2) \pq and \pq K_2(u) \ge \min(\ka_1(\de),\ka_0\|\na u\|_{L^2}^2).}
\end{lem}
\begin{proof}
$\kappa_{0}$ is an absolute constant that will be determined via the constant in~\eqref{GN}. 
First we prove the conclusion separately for $s=0$ and $s=2$ by contradiction. Fix $s=0$ or $s=2$ and $\de>0$, and 
suppose that there exists a sequence $\vec u_n\in \HH$ satisfying \eqref{energy region} with $\e_{0}=1/n$ but neither \eqref{-K bd} 
nor \eqref{+K bd} with $\ka_1=1/n$. Since $K_s(u_n)$ is bounded, the definition \eqref{def G} implies that $G_s(u_n)\simeq\|u_n\|_{H^1}^2$ is also bounded, and so $K_s(u_n)\to 0$. Then by the same argument as in Lemma~\ref{minimization}, we deduce that $u_n$ converges, after extraction of a subsequence, strongly to $0$ or $\pm Q$. In the latter case, \eqref{energy region} implies that 
\EQ{\nn 
 \de^2 \le \liminf_{n\to\I} \|\dot u_n\|_{L^2}^2 \le 2\e_0^2,}
which is precluded by choosing $\e_0(\de)<\de/2$. If $u_n\to 0$, then \eqref{GN} implies that the quadratic part dominates in $K_s(u_n)$ for large $n$, so \eqref{+K bd} holds for some $\ka_0>0$ independently of $\de$. Thus we obtain \eqref{-K bd} or \eqref{+K bd}, separately for $s=0$ and for $s=2$. 

It remains to show that they have the same sign. First note that 
\EQ{\nn 
 \pt \ti\K_s^+:=\{\vec u\in\HH\mid E(\vec u)<J(Q)+\e_0^2,\ d_Q(\vec u)>\de,\ K_s(u)\ge 0\},
 \pr \ti\K_s^-:=\{\vec u\in\HH\mid E(\vec u)<J(Q)+\e_0^2,\ d_Q(\vec u)>\de,\ K_s(u)< 0\},}
for $s=0,2$ are open sets satisfying 
\EQ{\nn
 \ti\K_s^+\cap\ti\K_s^-=\emptyset, \pq \ti\K_0^+ \cup \ti\K_0^- = \ti\K_2^+\cup \ti\K_2^-.}
Since $\ti K_0^+$ and $\ti K_2^+$ have the point $0$ in common, it suffices to show that both are connected. 
For that purpose, we use two kinds of deformations in $\ti\K_s^+$. Fix $s=0,2$ and take any $u\in \ti\K_s^+$ satisfying $\de<d_Q(\vec u)\le 2\de\ll \de_E$. Recall the expansion 
\EQ{\nn 
 \pt 2d_Q(\vec u)^2=k^2\la^2+\LR{L_+\ga|\ga}+\|\dot v\|_{L^2}^2-2C(v),
 \pr 2[E(\vec u)-J(Q)]=-k^2\la^2+\LR{L_+\ga|\ga}+\|\dot v\|_{L^2}^2-2C(v).}
Lemma \ref{1st exit} implies $|\la|\simeq d_Q(\vec u)$ provided that we choose $\e_0(\de)^2<\de^2/2$. 
We deform $\vec u$ by increasing $|\la|$, while fixing $\ga$ and $\dot u$. Then $E(\vec u)$ decreases and $d_Q(\vec u)$ increases, as long as $d_Q(\vec u)\ll \de_E$. Meanwhile, $\vec u$ remains in $\ti\K_s^+$ and eventually 
\EQ{\nn
 E(\vec u)-J(Q)\lec -d_Q(\vec u)^2 + O(\de^2) + o(d_Q(u)^2) \ll -\de^2.}
Thus we can deform $\ti\K_s^+\cap\{d_Q(\vec u)\le 2\de\}$ into 
\EQ{\nn 
 \{\vec u \in \HH\mid E(\vec u)-J(Q)\ll-\de^2, K_s(u)\ge 0\}\subset\ti\K_s^+,}
which is contracted to $\{0\}$ by the scaling transform as in Lemma \ref{sign K below Q}
\EQ{\nn 
 \vec u=(u,\dot u) \mapsto (u_s^\nu,\nu\dot u) \pq(\nu:1\to+0).}
For the remaining part of $\ti\K_s^+$ we also use this scaling transform, until either reaching $0$, or hitting the sphere $d_Q(\vec u)=2\de$, where it is reduced to the previous case. Thus we conclude that $\ti\K_s^+$ for both $s$ are connected and coincide. 
\end{proof}

\subsection{Sign function away from the ground states}
The above two lemmas enable us to define a sign functional away from $\pm Q$ by combining those of $-\la$ and $K_s$. 
\begin{lem} \label{sign}
Let $\de_S:=\de_X/(2C_*)>0$ where $\de_X$ and $C_*\ge 1$ are constants from Lemma \ref{2nd exit}. Let $0<\de\le\de_S$ and 
\EQ{ \label{def HS}
 \HH_{(\de)}:=\{\vec u\in\HH \mid E(\vec u)<J(Q)+\min(d_Q(\vec u)^2/2,\e_0(\de)^2)\},} 
where $\e_0(\de)$ is given by Lemma \ref{K lower bd}. Then there exists a unique continuous function $\Sg:\HH_{(\de)}\to\{\pm 1\}$ satisfying 
\EQ{ \label{def Sg}
 \CAS{\vec u\in\HH_{(\de)},\ d_Q(\vec u)\le\de_E &\implies  \Sg(\vec u)=-\sign\la,\\
 \vec u\in\HH_{(\de)},\ d_Q(\vec u)\ge\de &\implies \Sg(\vec u)=\sign K_0(u)=\sign K_2(u),}}
where we set $\sign 0=+1$ (a convention for the case $u=0$). 
\end{lem}
\begin{proof}
Lemma \ref{1st exit} implies that $\sign \la$ is continuous for $d_Q(\vec u)\le\de_E$, and Lemma \ref{K lower bd} implies that $\sign K_0(\vec u)=\sign K_2(\vec u)$ is continuous for $d_Q(\vec u)\ge\de$. Hence, it suffices to see that they coincide at $d_Q(\vec u)=\de_S\in[\de,\de_X]$ in $\HH_{(\de)}$. Let $u$ be a solution of NLKG with $\vec u(0)\in\HH_{(\de)}$ and $d_Q(\vec u(0))=\de_S$. Then Lemma \ref{2nd exit} implies that $\vec u(t)$ stays in $\HH_{(\de)}$ and $\sign\la(t)$ is constant, until $d_Q(\vec u(t))$ reaches $\de_X$, which is after $-\sign K_s(u(t))$ becomes the same as $\sign\la(t)$, because $2C_*\de_S\le\de_X$. Since $\sign K_s(\vec u)$ is constant for $d_Q(\vec u)\ge\de$, we conclude that $\sign\la(t)=-\sign K_s(u(t))$ from the beginning $t=0$. 
\end{proof}

The $\Sg=+1$ side is uniformly bounded in the energy, as the following lemma shows. 
 
\begin{lem} \label{Sg+ eng bd}
There exists $M_*\sim J(Q)^{1/2}$ such that for any $\vec u\in\HH_{(\de_S)}$ satisfying $\Sg(\vec u)=+1$ we have $\|\vec u\|_\HH \le M_*$. 
\end{lem}
\begin{proof}
If $d_Q(\vec u)\le\de_S$, then 
\EQ{
 \|\vec u\|_\HH \le \|Q\|_{H^1}+\|\vec v\|_\HH \lec J(Q)^{1/2}+d_Q(\vec u)^{1/2} \lec J(Q)^{1/2}.}
If $d_Q(\vec u)\ge\de_S$, then $K_0(u)\ge 0$ together with \eqref{def G} implies that 
\EQ{
 \|\vec u\|_\HH^2 = 4E(\vec u)-K_0(u) \le 4E(\vec u) \le 4[J(Q)+\e_0(\de_S)^2] \lec J(Q)}
 as desired. \end{proof}

\subsection{Vanishing kinetic energy leads to scattering}
Finally, we encounter the following problem in using the localized virial identity \eqref{loc virial*}: in the regime $K_2(u(t))\ge0$ this functional can become arbitrarily small  around $0$ for the $u$ component of $\HH$. 
That is, $K_2(u(t))$ can vanish at some time if (and only if) $\|\na u(t)\|_{L^2_x}$ does so, see Lemma \ref{K lower bd}. 
Notice that we should treat this kind of vanishing only in the time averaged sense. The idea is then that all frequencies have to shift to $0$, which leads to the scattering in both directions by the small Strichartz norm of subcritical regularity. 

\begin{lem} \label{0freq scat}
For any $M>0$, there exists $\mu_0(M)>0$ with the following property. Let $u(t)$ be a finite energy solution of NLKG \eqref{eq:NLKG} on $[0,2]$ satisfying 
\EQ{ \label{low freq conc} 
 \|\vec u\|_{L^\I_t(0,2;\HH)} \le M, \pq 
 \int_0^2 \|\na u(t)\|_{L^2}^2 \, dt \le \mu^2}
for some $\mu\in(0,\mu_0]$. Then $u$ extends to a global solution and scatters to $0$ as $t\to\pm\I$, and moreover $\|u(t)\|_{L^3_tL^6_x(\R\times\R^3)}\ll\mu^{1/6}$. 
\end{lem}
\begin{proof}
First we see that $u$ can be approximated by the free solution 
\EQ{\nn 
 v(t):=e^{i\LR{\na}t}v_+ + e^{-i\LR{\na}t}v_-,
 \pq v_\pm:=[u(0)\mp i\LR{\na}^{-1}\dot u(0)]/2.}
This follows  simply from the Duhamel formula 
\EQ{\nn 
 \|v-u\|_{L^\I_tH^1_x(0,2)} \pt\lec \|u^3\|_{L^1_tL^2_x(0,2)}
 \lec \|u\|_{L^3_tL^6_x}^3 
 \pr\lec \|\na u\|_{L^2_tL^2_x(0,2)}^2\|\na u\|_{L^\I_tL^2_x(0,2)} \le \mu^2 M \ll \mu,}
if $\mu_0 M\ll 1$, where we used H\"older's inequality  and the Sobolev embedding $\dot H^1\subset L^6$. In particular, 
\EQ{\nn 
 4\mu^2 \pt\ge \int_0^2\|\na v(t)\|_{L^2_x}^2 \, dt
 \pr=C\int|\x|^2\left[2|\hat v_+|^2+2|\hat v_-|^2+\Im\{\LR{\x}^{-1}(e^{4i\LR{\x}}-1)\hat v_+\bar{\hat v_-}\}\right]d\x
 \pr\gec\|\na v_+\|_{L^2}^2+\|\na v_-\|_{L^2}^2,}
where $\hat v$ denotes the Fourier transform in $x$ of $v$. Now we use the Strichartz estimate for the free Klein-Gordon equation
\EQ{\label{eq:KGS}
 \|e^{\pm i\LR{\na}t}\fy\|_{L^3_tB^{4/9}_{18/5,2}(\R\times\R^3)} \lec \|\fy\|_{H^1_x},}
where $B_{p,q}^s$ denotes the Besov space with $s$ regularity on $L^p$. Using the terminology of~\cite[Lemma~4.1]{IMN}, \eqref{eq:KGS}
means that $(\frac13, \frac{5}{18}, \frac{4}{9})$ is $1$-admissible. 
Combining it with Sobolev, we obtain
\EQ{\nn 
 \|v\|_{L^3_tL^6_x(\R\times\R^3)} \pt\lec \|v\|_{L^3_t\dot B^{1/3}_{18/5,2}(\R\times\R^3)} \pr\lec \sum_\pm\|v_\pm\|_{\dot H^{1/3}\cap\dot H^{8/9}}
 \lec M^{2/3}\mu^{1/3}+M^{1/9}\mu^{8/9} \ll \mu^{1/6},}
if $\mu_0 M^4\ll 1$. Therefore,  we can identify $u$ as the fixed point for the iteration in the global Strichartz norm 
\EQ{\nn 
 \|u\|_{L^\I_t H^1_x(\R\times\R^3)} \lec M, \pq \|u\|_{L^3_tL^6_x(\R\times\R^3)} \ll \mu^{1/6},}
which automatically scatters. 
\end{proof}

\subsection{Local virial identity and non-existence of almost homoclinic orbits}Using the constants in Lemmas \ref{2nd exit}--\ref{0freq scat}, we choose $\e_*,\de_*,R_*,\mu>0$ such that 
\EQ{ \label{choice e R*}
 \pt \de_*\le\de_S,\pq \de_*\ll \de_X, \pq \e_*\le\e_0(\de_*), 
 \pr \e_* \ll R_*\ll\min(\de_*,\ka_1(\de_*)^{1/2},\ka_0^{1/2}\mu,J(Q)^{1/2}),}
\EQ{ \label{choice mu}
 \mu<\mu_0(M_*), \pq \mu^{1/6}\ll J(Q)^{1/2}.} 

Suppose that a solution $u(t)$ on the maximal existence interval $I\subset\R$ satisfies for some $\e\in(0,\e_*]$, $R\in(2\e,R_*],$ and $\t_1<\t_2<\t_3\in I$, 
\EQ{\nn 
  E(\vec u)< J(Q)+\e^2, \pq d_Q(\vec u(\t_1))<R<d_Q(\vec u(\t_2))>R> d_Q(\vec u(\t_3)).}
Then there exist $T_1\in(\t_1,\t_2)$ and $T_2\in(\t_2,\t_3)$ such that 
\EQ{\nn 
 \pt d_Q(\vec u(T_1))=R=d_Q(\vec u(T_2))\le d_Q(\vec u(t))\pq (T_1<t<T_2).}
Lemma \ref{sign} gives us a fixed sign 
\EQ{
 \{\pm 1\}\ni \sg:=\Sg(u(t)) \pq (T_1<t<T_2).}

Now we derive the localized virial identity with a precise error bound. The cut-off function is defined by 
\EQ{\nn 
 w(t,x)=\CAS{\chi(x/(t-T_1+S)) & (t<(T_1+T_2)/2),\\ \chi(x/(T_2-t+S)) & (t>(T_1+T_2)/2),}}
where $S\gg 1$ is a constant to be determined later, and $\chi$ is a radial smooth function on $\R^3$ satisfying $\chi(x)=1$ for $|x|\le 1$ and $\chi(x)=0$ for $|x|\ge 2$. Using the equation we have 
\EQ{ \label{loc virial}
 V_w(t):=\LR{wu_t|(x\na+\na x)u}, \pq \dot V_w(t)=-K_2(u(t)) + O(E_{1}^0(t)),}
where $E_{j}^0(t)$ denotes the exterior energy defined by 
\EQ{\nn 
 \pt E_{j}^0(t):=\int_{X_{j}(t)}e^0(u) \,dx, \pq e^0(u):=[|\dot u|^2+|\na u|^2+|u|^2]/2, 
 \pr x \in X_{j}(t) \iff \CAS{|x|>j[t-T_1+S] &(T_1<t<\frac{T_1+T_2}{2}),\\
 |x|>j[T_2-t+S] &(\frac{T_1+T_2}{2}<t<T_2),}}
while the nonlinear version is denoted by 
\EQ{\nn 
 E_{j}(t):=\int_{X_{j}(t)}E(\vec u)\,dx, \pq E(\vec u):=e^0(u)-|u|^4/4.}
We infer from the finite propagation speed that 
\EQ{\nn 
 E_{1}(t)\le \max(E_{1}(T_1),E_{1}(T_2))
 \lec e^{-2S}+\sum_{t=T_1,T_2}\|\vec \ga(t)\|_E^2,}
where the term $e^{-2S}$ is dominating the tails of $Q$ and $\r$, due  to their exponential decay. Hence, choosing 
\EQ{\nn 
 S \gg |\log R| \gg 1,} 
we obtain $E_{1}(t)\lec R^2$. Similarly, we have 
\EQ{\nn 
 \sup_{j\ge 1}E_{j}(t) \lec R^2 \pq(T_1<t<T_2)}
To bound the free version, we use the exterior Sobolev inequality
\EQ{\nn 
 \|\fy\|_{L^4(r>S)} \lec \|\fy\|_{H^1(r>S)}.}
We remark that it does not require the radial symmetry. Then 
\EQ{ \label{E ext bd}
 E_{j}^0(t) \le E_{j}(t) + CE_{j}^0(t)^2 \pq (T_1<t<T_2),}
uniformly for $j\ge 1$. Since $E_{j}(t)\to 0$ as $j\to\I$ and $E_{j}(t)\lec R^2\ll 1$, the continuity in $j$ implies that  
\EQ{ \label{ext eng bd}
 E_{1}^0(t) \simeq E_{1}(t) \lec R^2 \pq (T_1<t<T_2).}
Thus we have obtained 
\EQ{ \label{eq:dotV}
 \dot V_w(t) = -K_2(u(t)) + O(R^2) \pq (T_1<t<T_2).}

We turn to the leading term $K_2$. In order to apply the ejection Lemma \ref{2nd exit}, we need the exiting property of the solution \eqref{exiting condition}. For that purpose, take any $t_m\in[T_1,T_2]$ where $d_Q(\vec u(t))$ attains a minimum in $t$ such that 
\EQ{
 (R\le)\ R_m:=d_Q(\vec u(t_m)) = \inf_{|t-t_m|<t_0,t\in[T_1,T_2]} d_Q(\vec u(t)) <\de_*.}
$T_1$ and $T_2$ obviously satisfy the above, but there may be numerous other minimum points if $\vec u(t)$ is circulating in the phase space, which indeed happens for the approximating ODE obtained by the projection onto the $\la$ component.

Applying Lemma \ref{2nd exit} to $u(t-t_m)$ and $u(t_m-t)$, as well as Lemma \ref{sign}, one obtains 
\EQ{ \label{u in Im}
 \pt d_Q(\vec u(t)) \simeq -\sg \la(t) \simeq e^{k|t-t_m|}R_m,
 \pq \sg K_2(u(t)) \gec d_Q(\vec u(t))-C_*R_m,}
until $d_Q(\vec u(t))$ reaches $\de_X\gg\de_*$. Let $I_m$ denote that time interval around $t_m$. 
Then the integral of \eqref{eq:dotV} is estimated on each $I_m$ by 
\EQ{ \label{int Im}
 [\sg V_w(t)]_{I_m} \gec \int_{I_m} [d_Q(\vec u(t))-C_*R_m-O(R^2)]dt
 \sim \de_X,}
thanks to the exponential growth and $R\le R_m\le\de_*\ll\de_X$. In the remainder 
\EQ{
 I':=[T_1,T_2]\setminus\Cu_m I_m,} 
one has  $d_Q(\vec u(t))>\de_*$, and $\e\le\e_0(\de_*)$. Hence, Lemma \ref{K lower bd} gives us \eqref{-K bd} if $\sg=-1$, or \eqref{+K bd} if $\sg=1$. 

In the latter case, Lemma \ref{Sg+ eng bd} implies that $\|\vec u\|_\HH\le M_*$ on $[T_1,T_2]$. Since $d_Q(\vec u(t))>\de_*\gg R$ for any $t\in I'$, the hyperbolic behavior \eqref{u in Im} on $I_m$ implies
\EQ{
 [t-1,t+1] \subset [T_1,T_2],}
and Lemma \ref{0freq scat} together with \eqref{choice mu} implies 
\EQ{ 
 \int_{t-1}^{t+1} \|\na u(s)\|_{L^2_x}^2 ds > \mu^2,}
since otherwise $\|u\|_{L^3_tL^6_x}\ll\mu^{1/6}\ll J(Q)^{1/2}$, which contradicts that $d_Q(\vec u(T_1))=R\ll J(Q)^{1/2}$. Combining it with \eqref{+K bd} or \eqref{-K bd}, we obtain 
\EQ{ \label{K2 int bd}
 \int_{t-1}^{t+1} K_2(u(s)) ds \gec \min(\ka_1(\de_*),\ka_0\mu^2) \gg R_*^2,}
due to the choice of $R$ in \eqref{choice e R*}. 
Combining this and \eqref{int Im}, we obtain 
\EQ{
 [\sg V_w(t)]_{T_1}^{T_2} \gec \de_X\times\#\{t_m\}.}
The left-hand side is bounded by using the exponential decay of $Q$, 
\EQ{\nn 
 \lec \sum_{t=T_1,T_2}\|\dot v(t)\|_{L^2}+S\|v(t)\|_E^2 \lec R+SR^2 \lec R,} 
provided that we choose $S$ such that 
\EQ{\nn 
 |\log R| \ll S \ll 1/R.}
Thus we have arrived at a contradiction since $R\ll\de_X$, precluding ``almost homoclinic" orbits, i.e., any trajectory which exits from and returns to the $R$ neighborhood of $\{\pm Q\}$. In other words, every solution is allowed to enter and exit a sufficiently small neighborhood of $\pm Q$ at most once. 

Notice that the contradiction simply means that $T_2=\I$, and then all the above analysis remains valid, except for the upper bound of $[\sg V_w]_{T_1}^{T_2}$. Thus we have proven the following result.

\begin{thm}[One-pass theorem] \label{no homo}
Let $\e_*,R_*>0$ be as in \eqref{choice e R*}. If a solution $u$ of NLKG on an interval $I$ satisfies for some $\e\in(0,\e_*]$, $R\in(2\e,R_*]$, and $\t_1<\t_2\in I$, 
\EQ{\nn
 E(\vec u)< J(Q)+\e^2, \pq d_Q(\vec u(\t_1))<R=d_Q(\vec u(\t_2)),}
then for all $t\in (\t_2,\I)\cap I=:I'$, we have $d_Q(\vec u(t))> R$. 
\end{thm}
Moreover, there exist disjoint subintervals $I_1,I_2,\dots\subset I'$ with the following property: On each $I_m$, there exists $t_m\in I_m$ such that 
\EQ{
 d_Q(\vec u(t))\simeq e^{k|t-t_m|}d_Q(\vec u(t_m)), \pq \min_{s=0,2}\sg K_s(u(t))\gec d_Q(\vec u(t))-C_*d_Q(\vec u(t_m)),}
where $\sg:=\Sg(\vec u(t))\in\{\pm 1\}$ is constant, $d_Q(\vec u(t))$ is increasing for $t>t_m$, decreasing for $t<t_m$, 
equals to $\de_X$ on $\p I_m$. For each $t\in I'\setminus\Cu_mI_m$ and $s=0,2$, one has  $(t-1,t+1)\subset I'$, $d_Q(\vec u(t))\ge\de_*$ and 
\EQ{ \label{int bd K}
 \int_{t-1}^{t+1}\min_{s=0,2}\sg K_s(u(t'))dt' \gg R_*^2.}

By the monotonicity, we can keep applying the above theorem at each $t>\t_2$ until $d_Q(\vec u)$ reaches $R_*$. 
Besides, one concludes that at  any later time $t_m>\t_2$ necessarily  $d_Q(\vec u)>R_*$. In other words, $u$ cannot return to the 
distance $R_*$ to $\pm Q$, after it is ejected to the distance $\de_X>R_*$.  

\section{Blowup after ejection}
Here we prove that the solution $u$ with $\Sg(\vec u(\t_2))=-1$ in Theorem \ref{no homo} blows up in finite time after $\t_2$, by the contradiction argument of Payne-Sattinger, using $K_0$. Suppose that $u$ extends to all $t>\t_2$ and let $y(t):=\|u(t)\|_{L^2}^2$. From the NLKG we have 
\EQ{ \label{ytt}
 \ddot y = 2[\|\dot u\|_{L^2_x}^2+\sg K_0(u(t))].}
Applying the lower bound on $K_0$ in Theorem \ref{no homo} to the integral yields
\EQ{
 [\dot y]_{\t_2}^\I \gec \sum_{I_m}\de_X+\int_{I'}R^2 dt=\I,}
and so $y(t)\to\I$ as $t\to\I$. Then from \eqref{ytt}, 
\EQ{
 \ddot y \ge -8E(\vec u)+6\|\dot u\|_{L^2}^2+2\|u\|_{H^1}^2 \ge 6\|\dot u\|_{L^2}^2 \ge 3(\dot y)^2/(2y),}
for large $t$, where we used Cauchy-Schwarz for $\dot y=2\LR{u|\dot u}$. Hence,
\EQ{
 \p_t^2(y^{-1/2})=-(2y^{3/2})^{-1}[y\ddot y-3(\dot y)^2/2]\le 0,}
which contradicts that $y\to\I$ as $t\to\I$. Therefore, $u$ does not extend to $t\to\I$. 

\section{Scattering after ejection}\label{sec:scat} 

For the solution $u$ with $\Sg(\vec u(\t_2))=+1$ in Theorem \ref{no homo}, the forward global existence follows from the energy bound Lemma \ref{Sg+ eng bd}. 
We prove its scattering to $0$ for $t\to\I$, by the contradiction argument of Kenig-Merle, using $K_2$. 

Fix $\e\in(0,\e_*)$ and let $\UU(\e,R_*)$ be the collection of all solutions $u$ of NLKG on $[0,\I)$ satisfying 
\EQ{
 E(\vec u) \le J(Q)+\e^2, \pq d_Q(\vec u[0,\I))\subset [R_*,\I), \pq \Sg(\vec u[0,\I))=+1.}
Note that the first two conditions imply that $\vec u[0,\I)\subset\HH_{(\de_*)}$ so that we can use Lemma \ref{sign} to define $\Sg(\vec u)$. 
By the remark after Theorem \ref{no homo}, any solution with $\Sg=+1$ in that theorem will eventually satisfy the above conditions. 

For each $E>0$, let $M(E)$ be a uniform Strichartz bound defined by
\EQ{ \label{def ME}
 M(E):=\sup\{\|u\|_{L^3_tL^6_x(0,\I)} \mid u\in\UU(\e,R_*),\ E(\vec u)\le E\},}
where we chose the norm $L^3_tL^6_x$ to be an $H^1$ subcritical and non-sharp admissible Strichartz norm such that its finiteness implies scattering. 
We know by \cite{IMN} that $M(E)<\I$ for $E<J(Q)$. In fact, in that case a uniform bound holds for $L^3_tL^6_x(\R)$. In order to extend
 this property to $J(Q)+\e^2$, put 
\EQ{ \label{def Estar}
 E^\star=\sup\{E>0 \mid M(E)<\I\}}
and assume towards a contradiction  that 
\EQ{
 E^\star<J(Q)+\e^2.} 

We consider the nonlinear profile decomposition for any sequence $u_n\in\UU(\e,R_*)$ satisfying 
\EQ{ \label{def un}
 E(u_n)\to E^\star, \pq \|u_n\|_{L^3_tL^6_x(0,\I)}\to\I.} 
We are going to show that the remainder in the decomposition is vanishing and there is only one profile which is a critical element, i.e.,
\EQ{ \label{crit elm}
 u_\star\in\UU(\e,R_*),\pq E(u_\star)=E^\star, \pq \|u_\star\|_{L^3_tL^6_x(0,\I)}=\I.}
The decomposition is given as follows. First we have the {\it linear profile decomposition} of Bahouri-G\'erard~\cite{BaG}. 

\begin{prop}
Let $\psi_n$ be a sequence of free Klein-Gordon solutions bounded in $\HH$. Then after replacing it by a subsequence, 
there exist a sequence of free solutions $v^j$ bounded in $\HH$, and sequences of times $t_n^j\in\R$ such that for $v_n^j$ and $\ga_n^k$ defined by
\EQ{\nn 
  \psi_n(t) = \sum_{j<k} v_n^j(t) + \ga_n^k(t), \pq v_n^j(t) = v^j(t+t_n^j),}
we have for any $j<k$, $\vec\ga_n^k(-t_n^j) \to 0$ weakly in $\HH$ as $n\to\I$, 
\EQ{ \label{orth}
 \pt \lim_{k\to \I} \limsup_{n\to\I} \|\ga_n^k\|_{(L^\I_tL^4_x\cap L^3_t L^6_x)(\R\times\R^3)}=0,
 \pq\lim_{n\to\I} |t_n^j-t_n^k| = \I.}
\end{prop}

The orthogonality of $t_n^j$ and the weak vanishing of $\ga_n^k$ implies that 
\EQ{ \label{H1 orth}
 \limsup_{n\to\I} \Bigl|\|\vec\psi_n\|_{\HH}^2-\sum_{j<k}\|\vec v^j\|_{\HH}^2-\|\vec\ga_n^k\|_{\HH}^2 \Bigr|=0,}
where the $\HH$ norms are independent of $t$ because all components are free solutions, and in particular all of $v^j$ and $\ga_n^k$ are uniformly bounded in $\HH$. It is simpler than the original form of the Bahouri-G\'erard decomposition for the wave equation, because the translational symmetry does not occur here  by the radial assumption, and the frequency parameter is fixed by the subcriticality of the remainder estimate. 
\begin{proof}
Since $\ga_n^k$ is bounded in $H^1_x$, interpolation with the Strichartz bound implies that it suffices to estimate the remainder in $L^\I_tL^4_x$. 
Let $\ga_n^0:=\psi_n$ and $k=0$. If 
\EQ{\nn 
 \nu^k:=\limsup_{n\to\I}\|\ga_n^k\|_{L^\I_t L^4_x}=0,} 
then we are done by putting $\ga_n^l=\ga_n^k$ for all $l>k$. 
Otherwise, there exists a sequence $t_n^k\in\R$ such that $\|\ga_n^k(-t_n^k)\|_{L^4_x} \ge \nu^k/2$ for large $n$. Since $\vec\ga_n^k(-t_n^k)\in \HH$ is bounded, after extracting a subsequence it converges weakly in $\HH$, and $\ga_n^k(-t_n^k)$ converges strongly in $L^4_x$. Let $v^k$ be the free solution given by the limit
\EQ{\nn 
 \lim_{n\to\I}\vec\ga_n^k(-t_n^k) = \vec v^k(0),}
then by Sobolev $\|v^k(0)\|_{H^1}\gec\nu^k$. We repeat the same procedure by induction for $k=1,2,3,\dots$. 
Let $U(t)$ denote the free Klein-Gordon propagator in $\HH$. If $|t_n^j-t_n^k|\to c\in\R$ for some $j<k$, then 
\EQ{\nn 
 \vec\ga_n^k(-t_n^k)=U(t_n^j-t_n^k)\vec\ga_n^k(-t_n^j) \to 0,}
weakly in $\HH$, hence $|t_n^j-t_n^k|\to\I$ as long as $v^k\not=0$. Then for all $j\le k$, 
\EQ{\nn 
 \vec\ga_n^{k+1}(-t_n^j)=\vec\ga_n^k(-t_n^j)-\vec v^k(t_n^k-t_n^j) \to 0}
weakly in $\HH$. In particular we have \eqref{H1 orth}, and so 
\EQ{\nn 
 \limsup_{n\to\I}\|\psi_n\|_{\HH}^2 \ge \sum_{j<k}\|v^j\|_{\HH}^2
 \gec \sum_{j<k}(\nu^j)^2,}
uniformly in $k$. Hence, $\limsup_{n\to\I}\|\ga_n^k\|_{L^\I_t L^4_x}=\nu^k \to 0$, as $k\to\I$. 
\end{proof}

Before applying the profile decomposition, we translate $u_n$ in $t$ to achieve 
\EQ{ \label{init adjust}
 d_Q(\vec u_n(0))>\frac{2}{3}\de_X, \pq K_2(u_n(0))\gg \e_*^2.} 
Since $d_Q(\vec u_n)$ remains  above $R_*$, the ejection Lemma \ref{2nd exit} implies that there exists $0\le T_n\lec k^{-1}\log(\de_X/R_*)$ so that $d_Q(\vec u_n(T_n))\ge\de_X$. Since $\Sg=+1$, Lemma \ref{Sg+ eng bd} implies that $\|u_n\|_{L^\I\HH(0,\I)}\le M_*$. Since $\|u_n\|_{L^3_tL^6_x}\to\I$, by the same argument as for \eqref{K2 int bd}, we deduce that there exists $0\le T_n'$ near $T_n$ such that 
\EQ{
 K_2(u_n(T_n')) \gg R_*^2 > 2\e_*^2, \pq d_Q(\vec u_n(T_n'))>\frac{2}{3}\de_X.}
Translating $u_n:=u_n(t-T_n')$, we obtain \eqref{init adjust}, in addition to \eqref{def un}. 

Now apply the above lemma to the free solution with the same initial data as $u_n$, and let $w_n^j$ be the nonlinear solution with the same data as $v_n^j$ at $t=0$
\EQ{\nn 
 U(t)\vec u_n(0)=\sum_{j<k} \vec v_n^j + \vec \ga_n^k, \pq \vec w_n^j(t)=U^N(t)\vec v_n^j(0),}
where $U^N(t)$ denotes the nonlinear Klein-Gordon propagator in $\HH$. 
Let $t_n^j\to t_\I^j\in[-\I,\I]$ and $u^j$ be the nonlinear solution satisfying 
\EQ{\nn
 \lim_{t\to t_\I^j}\|\vec u^j(t)-\vec v^j(t)\|_{\HH}=0,}
which exists at least locally around $t=t_\I^j$, as the unique solution of either the Cauchy problem at $t_\I^j\in\R$ or the wave 
operator at $t_\I^j\in\{\pm\I\}$. As a consequence of the local theory one has 
\EQ{\nn 
 w_n^j(t)-u^j(t+t_n^j)\to 0 \pq(n\to\I)}
in the energy and Strichartz norms locally around $t=0$. Thus we consider the following {\it nonlinear profile decomposition} 
\EQ{\nn 
 u_n = \sum_{j<k} u^j_n+\ga_n^k+error,\pq u^j_n:=u^j(t+t_n^j).}
Then $K_2(u_n(0))\gg \e^2$, $E(u_n)<J(Q)+\e^2$ and the orthogonality of the linear decomposition imply 
\EQ{\nn 
 \pt J(Q)-\e^2>\limsup_{n\to\I}[E(u_n)-K_2(u_n(0))/3 ]
 \pr\ge \limsup_{n\to\I} G_2(u_n(0)) = \limsup_{n\to\I} \sum_{j<k}G_2(u^j_n(0))+G_2(\ga_n^j(0)).}
Since $G_2$ is positive definite, we obtain $\limsup_n G_2(u^j_n(0))<J(Q)$, 
 which implies, via the minimizing property \eqref{cnstr min}, that $K_s(u^j_n(0))\ge 0$ for large $n$ and $s=0,2$. 
In particular, they all have positive energy, and so by the orthogonality applied to the energy, we deduce that $E(u^j)<J(Q)$ 
except for at most one of them. Those nonlinear profiles scatter as $t\to\pm\I$. 

If all of the profiles scatter for $t\to\pm\I$, or more precisely if $\|u^j\|_{L^3_tL^6_x(\R)}<\I$, then the 
following long-time perturbation argument implies that the original solutions $u_n$ also scatter and remain bounded in the Strichartz norm uniformly in $n$.  
\begin{lem} \label{PerturLem} 
There are continuous functions $\nu_0,C_0:(0,\I)^2\to(0,\I)$ such that the following holds: 
Let $I\subset \R$ be an interval, $u,w\in C(I;\HH)$ satisfying for some $A,B>0$ and $t_0\in I$
\EQ{ \label{asm ebd}
  \|\vec u\|_{L^\I_t(I;\HH)} + \|\vec w\|_{L^\I_t(I;\HH)}  \le A, \pq \|w\|_{L^3_t(I;L^6_x)} \le B,} 
\EQ{\nn 
 \pt\|eq(u)\|_{L^1_t(I;L^2_x)} 
   + \|eq(w)\|_{L^1_t(I;L^2_x)} + \|\ga_0\|_{L^3_t(I;L^6_x)} \le \nu_0(A,B),}
where $eq(u):=\ddot u-\De u +u-u^3$, and $\vec\ga_0:=U(t-t_0)(\vec u-\vec w)(t_0)$.  Then
\EQ{ \nn 
  \|\vec u-\vec w-\vec\ga_0\|_{L^\I_t(I;\HH)}+\|u-w\|_{L^3_t(I;L^6_x)} \le C_0(A,B)\nu_0.} 
\end{lem}
\begin{proof}
Let $Z:=L^3_tL^6_x$ and 
\EQ{\nn 
 \ga:=u-w, \pq e:=(\p_t^2-\De+1)(u-w)-u^3+w^3.} 
There is a partition of the right half of $I$: 
\EQ{\nn 
 \pt t_0<t_1<\cdots<t_n,\pq I_j=(t_j,t_{j+1}),\pq I\cap(t_0,\I)=(t_0,t_n),
 \pr \|w\|_{Z(I_j)} \le \de \pq(j=0,\dots,n-1), \pq n\le C(B,\de).}
We omit the estimate on $I\cap(-\I,t_0)$ since it is the same by symmetry. 
Let $\ga_j(t):=U(t-t_j)\vec \ga(t_j)$. 
Then the Strichartz estimate applied to the equations of $\ga$ and $\ga_{j+1}$ implies 
\EQ{ \label{est S'}
 \pn\|\ga-\ga_j\|_{Z(I_j)} + \|\ga_{j+1}-\ga_j\|_{Z(\R)}
 \pt\lec \|(w+\ga)^3-(w)^3+e\|_{L^1_tL^2_x(I_j)}
 \pr\lec (\de+\|\ga\|_{Z(I_j)})^2\|\ga\|_{Z(I_j)}+\nu_0.}
Hence, by induction on $j$ and continuity in $t$, one obtains 
\EQ{ \label{S' iterate}
 \|\ga\|_{Z(I_j)} + \|\ga_{j+1}\|_{Z(t_{j+1},t_n)}
 \le C[\|\ga_j\|_{Z(t_j,t_n)}+\nu_0] \le (2C)^j\nu_0 \le (2C)^n\nu_0\ll\de,}
provided that $\nu_0(A,B)$ is chosen small enough. Repeating the estimate~\eqref{est S'} once more, we can bound the full Strichartz norms on $\ga$. 
\end{proof}

The point in using the above perturbation lemma is that the orthogonality in \eqref{orth}, i.e., $|t_n^j-t_n^k|\to\I$,  implies that 
\EQ{ \label{nonlinearity decomposed}
 (\sum_{j<k} u_n^j+\ga_n^k)^3 = \sum_{j<k} (u_n^j)^3 + o(1),}
in $L^1_tL^2_x$ as $n\to\I$ and $k\to\I$, as soon as each $u_n^j$ has globally finite $L^3_tL^6_x$ norm. The same argument works even if the norm is finite only on $(0,\I)$ or $(-\I,0)$, provided that the translation $t_n^j$ and the interval $I$ are placed such that the infinite part of the norm does not contribute. 
Hence at least one profile does not scatter. Let $u^0$ be the non-scattering profile. 

If $t_\I=+\I$, then by the construction of the nonlinear profile, $u^0$ scatters as $t\to\I$, leading to a uniform bound on $\|u_n\|_{L^3_tL^6_x(0,\I)}$ for large $n$, a contradiction. 

If $t_\I^0=-\I$, then $u^0$ scatters as $t\to-\I$ by the same reason. Hence $u^0$ does not scatter as $t\to+\I$. If $d_Q(\vec u^0(t))>3\e$ on its maximal existence interval, then $\Sg(\ti u^0(t))=+1$ is preserved from $t=+\I$, so $\ti u^0$  is a global solution and a critical element, after time translation if necessary. Indeed, if it enters the interval $(3\e,R_*]$, but no deeper than that, then the ejection takes place and so the solution can never return to the $R_*$ distance for later time. 
Suppose now that $d_Q(\vec u^0(t_*))\le 3\e$ for some $t_*\in\R$. Then $\|u^0\|_{L^3_tL^6_x(-\I,t_*)}<\I$ implies that \eqref{nonlinearity decomposed} holds in $L^1_tL^2_x(-\I,t_*-t_n^0)$. Hence, the nonlinear profile decomposition becomes a good approximation of $u_n$ on that interval. Then taking $k$ and $n$ large enough, we infer from the perturbation lemma 
\EQ{\nn 
 d_Q(\vec u_n(t_*-t_n^0))\le d_Q(\vec u^0(t_*))+O(\e) \lec \e \ll R_*,}
contradicting that $d_Q(\vec u_n(0,\I))\subset[R_*,\I)$, since $t_*-t_n^0\to\I$. 

The remaining case $t_\I^0\in\R$ is also similar. If $u^0$ scatters for $t\to\I$, then the perturbation lemma implies that $\|u_n\|_{L^3_tL^6_x(0,\I)}$ is bounded for large $n$. Hence, $u^0$ does not scatter for $t\to\I$. On the other hand, the decomposition at $t=0$ gives 
\EQ{\label{init u^0}
 \frac{2}{3}\de_X \le\limsup_{n\to\I}d_Q(\vec u_n(0)) \le d_Q(\vec u^0(t_\I^0)) + O(\e),} 
and so $d_Q(\vec u^0(t_\I^0))>\de_X/2>\de_S\gg R_*$. Moreover, 
\EQ{ 
 K_s(\vec u^0(t_\I^0))=\lim_{n\to\I}K_s(\vec u^0(t_n^0))\ge 0,} 
and so we deduce from Lemma \ref{sign} that $\Sg(\vec u^0(t_\I^0))=+1$. 

By the same argument as above, $u^0$ is either a critical element or $d_Q(\vec u^0(t_*))\le 3\e$ for some $t_*>t_\I^0$. Then \eqref{init u^0} together with the above theorem implies that $d_Q(\vec u^0(t))\ge R_*$ for $t\le t_\I^0$. If $u^0$ does not scatter as $t\to-\I$, then $u^0(t_\I^0-t)$ is a critical element. Suppose that $u^0$ scatters as $t\to-\I$. Then the profile decomposition is a good approximation of $u_n$ on $(-\I,t_*-t_n^0)$ for large $n$. Then, by the perturbation lemma for large $k$ and $n$, 
\EQ{\nn 
  \e \gec d_Q(\vec u^0(t_*))+O(\e) \ge \limsup_{n\to\I}d_Q(\vec u_n(t_*-t_\I^0)) \ge R_*,}
which contradicts $\e\ll R_*$. 

Thus in conclusion, for some $T\in\R$ and $\sg\in\{\pm 1\}$, $u^0(\sg t+T)$ is a critical element. Since $u_n$ is a minimizing sequence, it implies also that the other components must vanish strongly in the linear profile decomposition. 
Therefore, on a subsequence, 
\EQ{
 \lim_{n\to\I}\|\vec u_n(T_n')-\vec u^0(t_n^0)\|_{\HH}=0,}
where $T_n'\ge 0$ is the time shift for \eqref{init adjust}. Both $T_n'$ and $t_n^0$ are bounded from above as $n\to\I$. If $t_n^0\to-\I$ then $u^0$ scatters for $t\to-\I$, and so $\|u_n\|_{L^3_tL^6_x(-\I,T_n')}$ is bounded for large $n$, by the local theory of the wave operator. 

Applying this to the sequence of solutions $u_n:=u^0(t+\t_n)$ for arbitrary $\t_n\to\I$, one obtains the precompactness of the forward trajectory of the critical element, and a contradiction from a localized (time-independent) virial identity together with the lower bound on $K_2$. These steps are simpler than those in Kenig-Merle~\cite{KM2} (see \cite{IMN} for NLKG). 
Thus we conclude that no solution $u$ satisfies \eqref{crit elm}, and therefore $E^\star=J(Q)+\e^2$.

\begin{thm}
For each $\e\in(0,\e_*]$, there exists $0<M(J(Q)+\e^2)<\I$ such that if a solution $u$ of NLKG on $[0,\I)$ satisfies $E(\vec u)\le J(Q)+\e^2$, $d_Q(\vec u(t))\ge R_*$ and $\Sg(\vec u(t))=+1$ for all $t\ge 0$, then $u$ scatters to $0$ as $t\to\I$ and $\|u\|_{L^3_tL^6_x(0,\I)}\le M$. 
\end{thm}
Note that the uniform Strichartz bound is valid only on the time interval where the solution is already away from $\pm Q$ by a fixed distance. 
Without it, one cannot have any uniform bound even for those solutions scattering for both  $t\to\pm\I$, since they can stay close to $\pm Q$ for arbitrarily long time. 

\section{Classification of the global behavior, proof of Theorem~\ref{thm:main}}
Fix $0<\e\le\e_*$ and let $\HH^\e:=\{\vec u\in \HH \mid E(\vec u)<J(Q)+\e^2\}$ 
be the initial data set. 
We can define the following subsets according to the global behavior of the solution $u(t)$ to the NLKG equation: for $\si=\pm$ respectively, 
\EQ{\label{eq:STB def}
 \pt \S_\si^{\e}=\{\vec u(0)\in\HH^\e \mid \text{$u(t)$ scatters as $\si t\to\I$}\},
 \pr \T_\si^{\e}=\{\vec u(0)\in\HH^\e \mid   \text{$u(t)$ trapped by $\{\pm Q\}$ for $\si t\to\I$}\}, 
 \pr \B_\si^\e=\{\vec u(0)\in\HH^\e \mid   \text{$u(t)$ blows up in $\si t>0$}\}.}
The trapping for $\T_+^\e$ can be characterized as follows, with any $R\in(2\e,R_*)$: 
\EQ{\nn 
 \exists T>0, \pq \forall t>T,\pq d_Q(\vec u(t))<R.}
Obviously those sets are increasing in $\e$, and have the conjugation property
\EQ{\nn 
 X_\mp^\e = \{(u(0),-\dot u(0))\in\HH^\e \mid \vec u(0)\in X_\pm^\e\},}
for $X=\S,\T,\B$. Moreover, $\S_+$ and $\T_+$ are forward invariant by the flow of NLKG, while $\S_-$ and $\T_-$ are backward invariant. We have proven in the previous sections that 
\EQ{\nn 
 \HH^\e = \S_+^\e \cup \T_+^\e \cup \B_+^\e = \S_-^\e \cup \T_-^\e \cup \B_-^\e,}
the disjoint union for each sign. 
It follows from the scattering theory that $\S_\pm^\e$ are open. 
We claim the same for $\B_\pm^\e$, which is not a general fact. 
By the energy estimate 
\EQ{\nn 
 \|\vec u(t)\|_{L^\I_t \HH(0,T)} \lec \|\vec u(0)\|_\HH+T\|u\|_{L^\I_t H^1_x(0,T)}^3,}
we deduce that if $t=T^*$ is the blow-up time then 
\EQ{\nn 
 \|u(t)\|_{\HH} \gec |T^*-t|^{-1/2}.}
Hence, from the integral inequality 
\EQ{\nn 
 \p_t^2\|u\|_{L^2_x}^2 = 2[\|\dot u\|_{L^2_x}^2-K_0(u(t))]
 \ge 6\|\dot u\|_{L^2_x}^2+2\|u\|_{H^1_x}^2-8E(\vec u),}
we deduce that $\p_t\|u\|_{L^2_x}^2=2\LR{u|\dot u}\to\I$  as well as $K_{0}(u(t))\to-\infty$ as $t\to T^*-0$. 
We claim that if $T^{**}<T^{*}$ is very close to blow-up time, then every solution starting in $B_{1}(u(T^{**}))$,
the unit-ball around $u(T^{**})$ relative to~$\HH$, necessarily blows up in the positive time direction. By the
 Payne-Sattinger argument, we only need to show that for any such solution,
$K_0(u(t))\le -\kappa<0$ for as long as it is defined.  It is clear that this condition will hold initially. 
By Lemma \ref{K lower bd} with the choice of $\e_*$ in \eqref{choice e R*}, we further see that it can be violated only if the solution returns to a neighborhood of $\pm Q$ of size $\de_S$. 
However, in that case necessarily $|\LR{u|\dot u} |\lec 1$, which is impossible since $\LR{u|\dot u}$ starts off very large and has to increase as long as $K_{0}(u(t))<0$. 
Thus, any solution close to $u$ around $t=T^{**}$ cannot come close to $\pm Q$ and has to blow up sooner or later, as claimed. 
Therefore,  $\B_\pm^\e$ are also open, so $\T_\pm^\e$ are relatively closed in $\HH^\e$. 

Since $\B^\e_+$ and $\S_+^\e$ are disjoint open, they are separated by $\T_+^\e$, that is, any two points from $\S_+^\e$ and $\B_+^\e$ cannot be joined by a curve without passing through $\T_+^\e$. 

In a small ball around $\pm Q$, it is easy to see by means of the linearized flow that the open intersections $\B^\e_\pm\cap \S_\mp^\e$, $\B^\e_+\cap \B^\e_-$ and $\S_+^\e\cap \S_-^\e$ are all non-empty for any $\e>0$. To obtain a solution $u$ in $\B_+^\e\cap \S_-^\e$, choose  initial data such that 
\EQ{\nn 
 \la(0)=0, \pq \dot\la(0)=k \th \e, \pq \vec\ga(0)=0,}
for some $0<\th \ll 1$. Then the energy constraint $E(\vec u)<J(Q)+\e^2$ is satisfied, and by the same argument as in Lemma \ref{2nd exit}, we have 
\EQ{\nn 
 |\la(t)-\la_0(t)| \lec e^{2k|t|}\th^2\e^2, \pq \|\ga(t)\|_E \lec \th\e+e^{2k|t|}\th^2\e^2,}
as long as $e^{k|t|}\th\e\ll\de_X$, where the free solution $\la_0$ is given in this case
\EQ{\nn 
 \la_0(t)=\sinh(kt)\th\e.}
Hence, $u(t)$ exits the $R_*$ neighborhood of $\pm Q$ both in $\pm t>0$. Since $\la(t)t>0$ for $\th\e\ll e^{k|t|\th}\th\e\ll\de_X$, we obtain $\Sg(\vec u(t))t<0$ after exiting, and so $u$ blows up in $t>0$ and scatters for $t\to-\I$. 
Therefore, $\B_\pm^\e\cap \S_\mp^\e$ are both non-empty. 

In the same way, we construct solutions in $\B_+^\e\cap \B_-^\e$ and $\S_+^\e\cap \S_-^\e$, respectively. 
Let $u^{(\pm)}$ be those solutions for which $\la$ are approximated by the free solutions 
\EQ{\nn 
 \la_0^{(\pm)}(t)=\pm\cosh(kt)\th\e.}
Then $u^{(+)}\in \B_+^\e\cap \B_-^\e$ and $u^{(-)}\in \S_+^\e\cap \S_-^\e$. 

For $\e \ll e^{kt}\th\e \ll \e^{1/2}$, the distances from the above solution $u\in \B_+^\e\cap \S_-^\e$ and from $\{\pm Q\}$ are estimated by 
\EQ{\nn 
 \|\vec u(t)-\vec u^{(+)}(t)\|_E \lec e^{2kt}\th^2\e^2+\th\e \ll \e,
 \pq d_Q(\vec u(t))\gec e^{kt}\th\e \gg \e }
Hence, there exists a curve $\Ga\subset \HH^\e$ joining $u(t)$ and $u^{(+)}(t)$ within the region $\Sg(\vec u)<0$ and $d_Q(\vec u)\gg\e$. Since $u(t)\in \S_-^\e$ and $u^{(+)}(t)\in \B_-^\e$, there exists a point $p_{0}\in \T_\e^-\cap \Ga$. Since the solution starting from $p_{0}$ enters the $3\e$-ball around $\pm Q$ as $t\to-\I$, and initially $p_{0}$ is much further away and $\Sg(p_{0})<0$, we conclude by the one-pass theorem that $p\in \B^\e_+$. Hence, $\T^\e_-\cap \B^\e_+$ is non-empty as well. In the same way, we can find a point on the curve connecting $u(t)$ and $u^{(-)}(t)$ for some $t<0$, which is in $\T^\e_+\cap \S_-^\e$. Therefore, $\T^\e_\pm\cap \B^\e_\mp$ and $\T^\e_\pm\cap \S_\mp^\e$ are both not empty. Taking the limit $\e\to+0$, it is easy to observe that they contain infinitely many points on different energy levels. 

Finally, $\T^\e_+\cap \T^\e_-\ni Q$ is of course not empty, but there are infinitely many points besides $Q$, 
which can be seen by restricting $\HH^\e$ to the hyperplane $\dot u=0$: 
\EQ{\nn 
 \HH^\e_{\dot u=0}:=\{(u,0)\mid J(u)<J(Q)+\e^2\}\ni u^{(\pm)}(0).}
Since $u^{(+)}(0)\in \B_\e^+$ and $u^{(-)}(0)\in \S_\e^+$, and there are infinitely many connecting curves in $\HH^\e_{\dot u=0}$, there are infinitely many points in $\T^\e_+\cap\HH^\e_{\dot u=0}$ separating $\B^\e_+$ and $\S^\e_+$. The symmetry $u(t)\mapsto u(-t)$ implies that 
\EQ{\nn 
 \T^\e_+\cap\HH^\e_{\dot u=0} \subset \T^\e_+\cap \T^\e_-.}
Thus we have shown that the nine solution sets are all non-empty. 

By using the scattering theory, we can prove that $\S^\e_+$ and the energy section
\EQ{
 \S^{=\e}_+:=\{\vec u\in\HH\mid \text{$U^N(t)\vec u$ scatters as $t\to\I$ },\ E(\vec u)=J(Q)+\e^2\}}
are both pathwise connected. To see this, let $\vec u_j\in \S^\e_+$ or $\S^{=\e}_+$ for $j=0,1$ and let $\vec u_{j+}$ be their asymptotic profiles 
\EQ{
 \|U^N(t)\vec u_j-U(t)\vec u_{j+}\|_E \to 0 \pq(t\to\I).}
There exists a continuous curve $\vec u_{\th+}:[0,1]\to\HH$ such that 
\EQ{
 \|\vec u_{\th+}\|_\HH^2 = (1-\th)\|\vec u_{0+}\|_\HH^2+\th\|\vec u_{1+}\|_\HH^2,}
as well as $u_\th$ solving NLKG on some $(T_\th,\I)$ such that 
\EQ{
 \|\vec u_\th(t)-U(t)\vec u_{\th+}\|_E \to 0 \pq(t\to\I).}
Moreover, $T_\th$ and $\vec u_\th(T_\th+1)$ are also continuous in $\th$. The scattering property implies that $\vec u_\th(t)\in\S_+$ for any $t>T_\th$, and 
\EQ{
 E(\vec u_\th)=\|\vec u_{\th+}\|_\HH^2/2=(1-\th)E(\vec u_0)+\th E(\vec u_1).}
Let $T:=\sup_{0\le\th\le 1}T_\th +1<\I$, 
then $\vec u_j$ and $\vec u_j(T)$ are connected by the flow, while $\vec u_0(T)$ and $\vec u_1(T)$ are connected by $\vec u_\th(T)$, all included in $\S^\e_+$ or $\S^{=\e}_+$. This proves the connectedness of those sets. 

So far, all arguments of this section have been independent of the spectral gap property~\eqref{hyp:spec}. 
Now let us assume it for further investigation. Then $\T_+^\e$ contains the center-stable manifold for the $t\ge0$ direction, as constructed in Proposition~\ref{prop:mfd}.
In fact,  it consists of the maximal backward extension of all solutions on this manifold. The linear approximation is the hyperplane $k\la=-\dot\la$. Hence, the classification into those nine sets looks like $\otimes$ in the $(\la,\dot\la)$ phase plain around $Q$. Since the flow is $C^\I$ in both directions, $\T_\pm^\e$ are indeed connected, $1$-codimensional and smooth manifolds in the energy space. Moreover, a small ball around each point of $\T_+^\e$ is separated by the hypersurface $\T_+^\e$ into $\S_+^\e$ and $\B_+^\e$. 

Next we prove that $\T_+^\e$ consists of two connected components $\T_{+(\pm)}^\e$ defined by 
\EQ{
 \T_{+(\si)}^\e=\{\vec u\in\HH^\e \mid \text{$U^N(t)\vec u$ is trapped by $\si Q$ as $t\to\I$}\} .}
 To see that $\T_{+(+)}^\e$ is pathwise connected, take any couple of points and send them by the flow to the future where they are in $3\e$ distance from $Q$. Let $\vec u_j$ with $j=0,1$ be those two after translation, and let $\vec\ga_{j}$ be their asymptotic profiles 
\EQ{
 \|U^N(t)\vec u_j-\vec Q-U(t)\vec\ga_{j}\|_E\to 0 \pq (t\to\I), \pq E(\vec u_j)=J(Q)+\|\ga_{j}\|_\HH^2/2.}
There is a continuous $\vec\ga_{\th}:[0,1]\to\HH$ such that 
$\|\vec\ga_\th\|_\HH^2=(1-\th)\|\vec\ga_0\|_\HH^2+\th\|\vec\ga_1\|_\HH^2$. 
Let $\la_\th:=(1-\th)\LR{u_0-Q|\r}+\th\LR{u_1-Q|\r}$. Proposition \ref{prop:mfd} gives a continuous family of NLKG solutions $u_\th$ such that $\LR{u_\th(0)-Q|\r}=\la_\th$ and $\vec u_\th(t)-\vec Q-U(t)\vec\ga_\th\to 0$ as $t\to\I$. Then $\vec u_0$ and $\vec u_1$ are connected by the curve $\vec u_\th(0)$ in $\T_{+(+)}^\e$.

To see the separation of $\T_{+(+)}^\e$ from $T_{+(-)}^\e$, suppose for contradiction that $p$ is a common point of their boundaries. Then $p\in\T_+^{2\e}$ by its relative closedness, so $U^N(t)p$ scatters to $\si Q$ with $\si=\pm$. Then its small neighborhood can be mapped by the flow into a ball around $\si Q$. But it contains a point in $\T_{+(-\si)}^\e$, contradicting the one-pass theorem. Thus $\T_+^\e$ consists of two connected components $\T_{+(\pm)}^\e$. 

We can observe also that 
\EQ{
 \p\S_+^\e = \T_+^\e \cup \T_+^{=\e} \cup \S_+^{=\e}}
is connected. $\subset$ follows from that $\B_+^\e$ and $\{E(u)>J(Q)+\e^2\}$ are open. $\supset$ is easily seen by perturbing the data from the right. The connectedness follows from that of the following three sets 
\EQ{
 \T_{+(+)}^{<}\cup\T_{+(+)}^{=\e}, \pq \S_+^{=\e}\cup\T_+^{=\e}, \pq \T_{+(-)}^{<}\cup\T_{+(-)}^{=\e}.}

Kenig-Merle's solutions \cite{KM1,KM2} can be classified in our notation, see~\eqref{eq:STB def}, by 
\EQ{\nn 
  \S_+^{0}=\S_-^{0}, \pq \B_+^{0}=\B_-^{0}, \pq \T_+^{0}=\T_-^{0}=\emptyset}
(see \cite{IMN} for the case of NLKG). This statement essentially follows from our theorems, because the energy constraint $E(\vec u)<E(Q)$ prohibits the sign change of $K_s$ and scattering to $\pm Q$. The threshold case Theorem \ref{thm:threshold}, i.e., the analogue of Duyckaerts-Merle \cite{DM1,DM2}, requires a little more investigation in order to establish the uniqueness of the solutions in (5)-(9). 

\begin{proof}[Proof of Theorem \ref{thm:threshold}]
Let $\wt{X_\pm^0} = \Ca_{\e>0}X_\pm^\e$ for $X=\S,\B,\T$. Then Theorem \ref{thm:threshold} can be restated as 
\EQ{\nn 
 \pt  \wt{\S_+^0}\cap \wt{\B_-^0}= \wt{\S_-^0}\cap  \wt{\B_+^0}=\emptyset, \pq  \wt{\T_+^0}\cap \ \wt{\T_-^0}=\{\pm Q(t-s) \mid s\in\R\},
 \pr  \wt{\S_+^0} \cap  \wt{\T_-^0}=\{\pm W_-(t-s)\mid s\in\R\},
 \pq  \wt{\B_+^0}\cap  \wt{\T_-^0}=\{\pm W_+(t-s)\mid s\in\R\},
 \pr  \wt{\S_-^0}\cap  \wt{\T_+^0}=\{\pm W_-(s-t)\mid s\in\R\},
 \pq  \wt{\B_-^0}\cap  \wt{\T_+^0}=\{\pm W_+(s-t)\mid s\in\R\},}
for some solutions $W_\pm$. 
To see that $ \wt{\S_+^0}\cap  \wt{\B_-^0}=\emptyset$, let $u$ be a solution in $\wt{\S_+^0}\cap  \wt{\B_-^0}$. Theorem \ref{Lachesis} implies that there exists a finite interval $I_\e\subset\R$ for any $\e>0$, such that $\vec u^{-1}(B_{3\e}(\pm Q))=I_\e$. Then $I_\e$ is decreasing as $\e\to+0$, and hence by continuity of $\vec u$, there exists $t\in\R$ such that $\vec u(t)=\vec Q$, which implies $u\equiv Q$, a contradiction. 

Next, any solution $u$ in $ \wt{\T_+^0}$ corresponds to the case $\ga_\I\equiv 0$ in Proposition~\ref{prop:mfd}, and it  is therefore parametrized by $\la(0)\in\R$. Since $\la(t)\to 0$ as $t\to\I$, the uniqueness in the proposition implies that there are only three solutions converging to $Q$, modulo time translation: one with $\la>0$ decreasing (for large $t$), another with $\la<0$ increasing, and yet another with $\la\equiv 0$, i.e., $Q$. Let $W_\pm$ be the two solutions given by $\la(0)=\pm\nu/10$ and $\ga_\I\equiv 0$. 
Since $E(W_\pm)=J(Q)$ and $W_\pm\not=Q$, the sign of $K_0$ is fixed on their trajectories. Indeed, Lemma~\ref{2nd exit} applies to them backward from a neighborhood of $t=\I$, from which  we deduce that $\Sg(W_\pm(t))=\mp 1$ away from $t=\I$, and so $W_+\in \wt{ \B_-^0}$ and $W_-\in  \wt{\S_-^0}$. 
To prove that $W_\pm$ approach $Q$ exponentially, we return to~\eqref{inteq la red} and~\eqref{gaeq infty} with $\ga_\I=0$, estimating 
\EQ{
 \|(\la,\ga)\|_{X^k_T}:=\|e^{k\min(t-T,0)}\la(t)\|_{L^1\cap L^\I(0,\I)}+\|e^{k\min(t-T,0)}\ga(t)\|_{St(0,\I)}}
for $T\to\I$. The same argument as in the proof of Proposition \ref{prop:mfd} yields 
\EQ{
 \|(\la,\ga)\|_{X^k_T} \pt\lec e^{-kT}[|\la(0)|+\|(\la,\ga)\|_X^2] + \|(\la,\ga)\|_X\|(\la,\ga)\|_{X^k_T}
 \pr\lec e^{-kT}\nu + \nu \|(\la,\ga)\|_{X^k_T},}
and so $\|(\la,\ga)\|_{X^k_T} \lec e^{-kT}$, as desired. 
\end{proof}

\section*{Acknowledgments}
The second author wishes to thank Kyoto University, and especially Prof.\ Yoshio Tsutsumi, as well as Prof.\ Kenji Yajima from Gakushuin University, Tokyo,  
for their very generous hospitality. He was partially supported by a Guggenheim fellowship and the National Science Foundation, DMS--0653841. 

\appendix 
\section{Table of Notation}
{\small
\begin{longtable}{l|l|l}
 \hline 
 symbols & description & defined in \\
 \hline 
 $\HH,\HH^\e,\HH_{(\de)}$  & energy space and subsets& \eqref{def H},\eqref{def He},\eqref{def HS} \\
 \hline 
 $E(\vec u),J(u)$ & dynamic/static energy & \eqref{def E}, \eqref{def J} \\
 $K_s(u)$, $G_s(u)$& scaling derivatives of $J(u)$ & \eqref{def K},\eqref{def G} \\
 $N(v),N_\r(v),N_c(v)$ & higher order nonlinearity around $Q$ & \eqref{eq:KGL},\eqref{eq:Lgamma} \\
 $C(v)$ & higher terms in energy around $Q$ & \eqref{energy exp} \\
 $V_w(t)$ & localized virial & \eqref{loc virial} \\
 \hline 
 $\|\vec v\|_E$ & linearized energy norm & \eqref{def Enorm} \\
 $d_Q(\vec u)$ & nonlinear distance to $\pm Q$ & \eqref{def dQ} \\
 $\Sg(\vec u)$ & sign functional on $\HH_{(\de)}$ & Lemma \ref{sign} \\ 
 \hline 
 $L_+,\om$ & linearized operator around $Q$ & \eqref{def L+},\eqref{eq:Lgamma} \\
 $\r,k,P^+$ & ground state of $L_+$ & \eqref{def L+},\eqref{def P+} \\
 $v,\la,\ga,\la_\pm$ & components of $u$ around $\pm Q$ & \eqref{decop u},\eqref{def la+-} \\
 \hline 
 $\de_E$ & radius of energy nonlinearity & \eqref{def dE} \\
 $\de_X$ & terminal radius of ejection & Lemma \ref{2nd exit} \\
 $C_*$ & constant in a bound on $|K_s|$ & Lemma \ref{2nd exit} \\
 $\e_0(\de),\ka_0,\ka_1(\de)$ & constants for variational estimates & Lemma \ref{K lower bd} \\
 $\de_S$ & radius of variational estimates & Lemma \ref{sign} \\
 $M_*$ & energy bound in $\Sg=+1$ & Lemma \ref{Sg+ eng bd} \\
 $\e_*$ & energy threshold for one-pass theorem & \eqref{choice e R*} \\
 $\de_*$ & hyperbolic-variational threshold radius & \eqref{choice e R*} \\
 $R_*$ & exiting radius for one-pass theorem & \eqref{choice e R*} \\
 \hline 
\end{longtable}
}
Relation of the constants: 
\EQ{
 \pt 2C_*\de_S = \de_X \le \de_E \ll 1 \le C_*,  
 \pq \mu_0(M_*)>\mu\ll J(Q)^3, 
 \pr \e_*=\e_0(\de_S)<R_*/2\ll \min(\de_S,\ka_1(\de_S),\ka_0\mu^{1/2},J(Q)^{1/2}).}

\end{document}